\documentclass[letterpaper,12pt,reqno]{amsart}

\usepackage[margin=1in]{geometry}

\usepackage{graphicx,amsmath,amssymb,amsthm,color,tikz-cd}
\usepackage{bm}
\usetikzlibrary{decorations.pathreplacing}

\usepackage{hyperref}
\hypersetup{
    colorlinks   = true,
     citecolor    = black,
    linkcolor    = blue
}

\usepackage{float}

\newtheorem{thmintro}{Theorem}

\newtheorem*{remark*}{Remark}
\newtheorem*{sublemma}{Sublemma}

\numberwithin{equation}{section}
\newtheorem{theorem}{Theorem}[section]
\newtheorem{lemma}[theorem]{Lemma}
\newtheorem{proposition}[theorem]{Proposition}
\newtheorem{definition}[theorem]{Definition}
\newtheorem{corollary}[theorem]{Corollary}

\newcommand{\NE}{\operatorname{NE}}
\newcommand{\Sing}{\mathrm{Sing}}
\newcommand{\Reg}{\mathrm{Reg}}
\newcommand{\GGG}{\mathcal{G}}
\newcommand{\PPP}{\mathcal{P}}
\newcommand{\SSS}{S}
\newcommand{\proportion}{\eta}

\newcommand{\GS}{G\SSS}
\newcommand{\GtS}{G\tilde\SSS}
\newcommand{\CAT}{\mathrm{CAT}}

\newcommand{\eps}{\varepsilon}

\newcommand{\Addresses}{{
  \bigskip
  \footnotesize
	
  \textsc{Department of Mathematics, Statistics, and Computer Science, University of Illinois at Chicago, Chicago, IL}\par\nopagebreak
  \textit{E-mail address:} \texttt{bcall@uic.edu}
	
	\smallskip
	
	\textsc{Mathematics and Computer Science Department, Wesleyan University, Middletown, CT}\par\nopagebreak
  \textit{E-mail address:} \texttt{dconstantine@wesleyan.edu}
	
	\smallskip
	
	\textsc{Department of Mathematics, The University of Chicago, Chicago, IL}
	\par\nopagebreak
  \textit{E-mail address:} \texttt{aerchenko@stonybrook.edu}	
	\smallskip
	
	\textsc{Mathematics and Computer Science Department,}
	\par\nopagebreak
	\textsc{Southwestern University, Georgetown, TX}\par\nopagebreak
  \textit{E-mail address:} \texttt{sawyern@southwestern.edu}
	
	\smallskip
	
	\textsc{Department of Mathematics, University of Wisconsin-Madison, Madison, WI}\par\nopagebreak
  \textit{E-mail address:} \texttt{work2@wisc.edu}
}}

\title[Unique equilibrium states for flat surfaces with singularities]{Unique equilibrium states for geodesic flows on flat surfaces with singularities}
\author{Benjamin Call, David Constantine, Alena Erchenko,\\ Noelle Sawyer, Grace Work}
\date{}

\begin{document}

\maketitle

\vspace{0.5cm}

\begin{abstract}
Consider a compact surface of genus $\geq 2$ equipped with a metric that is flat everywhere except at finitely many cone points with angles greater than $2\pi$. Following the technique in the work of Burns, Climenhaga, Fisher, and Thompson, we prove that sufficiently regular potential functions have unique equilibrium states if the singular set does not support the full pressure. Moreover, we show that the pressure gap holds for any potential which is locally constant on a neighborhood of the singular set. Finally, we establish that the corresponding equilibrium states have the $K$-property, and closed regular geodesics equidistribute. 
\end{abstract}

%

\section{Introduction}

We examine the uniqueness of equilibrium states for geodesic flows on a specific class of CAT(0) surfaces, those where the negative curvature is concentrated at a finite set of points. Translation surfaces are examples of such surfaces. A translation surface $X$ is a pair $(X, \omega)$ where $X$ is a Riemann surface of genus $g$, and $\omega$ is a holomorphic one-form on $X$. The zeroes of this holomorphic one-form occur at a finite set of points. The one-form $\omega$ defines a metric which is flat everywhere except at its zeroes. At the zeroes the metric has a conical singularity with angle $2(n+1)\pi$, where $n$ is the order of the zero. For a more in-depth overview of translation surfaces see \cite{Wright, Zorich}.

In \cite{BCFT}, the authors prove that under certain conditions, a unique equilibrium state exists for potentials associated to the geodesic flow on a closed, rank-one manifold with nonpositive sectional curvature (an example of a CAT(0) space \emph{without} singularities). The conditions are a H\"older continuous potential and a pressure gap, that is, topological pressure of the flow restricted to the singular set is strictly less than pressure of the flow overall. The singular set they consider is all the vectors in the unit tangent bundle with rank larger than one.

When the singular set is empty -- for example in strictly negative curvature -- every H\"older potential has a unique equilibrium state. When the singular set is non-empty, an additional condition is necessary as the geodesic flow is nonuniformly hyperbolic. Restricting the pressure of the flow on the singular set is a way of describing the flow of the singular set as having a small enough impact on the system as a whole that uniqueness is still guaranteed.

The natural way to define a geodesic flow on CAT(0) surfaces is to look at the flow on the set of all geodesics (see Section~\ref{sec:defns}). Denote by $\GS$ the set of all geodesics on the surface $\SSS$ (see \eqref{def: GS}). 


In this paper, we study the uniqueness of equilibrium states for the geodesic flow described above (see Definition~\ref{defn: pressure}), as we are guaranteed existence for continuous potentials by entropy-expansivity of the flow (see Lemma~\ref{lem:entropy-expansive}). 
In particular, we use the technique of \cite{BCFT} in our setting and define the singular set to be the set of geodesics which never encounter any cone points or, when they do, turn by angle exactly $\pm \pi$. 

\begin{remark*}
Some other settings where the uniqueness of equilibrium states was studied are described in more detail below in the outline of the argument.
\end{remark*}


We prove the following:

\begin{thmintro}\label{thm: existence}
Let $g_t$ be the geodesic flow on  $\SSS$, a compact, connected surface of genus $\geq 2$ equipped with a metric that is flat everywhere except at finitely many cone points which have angle greater than $2\pi$. Let $\Sing$ be the singular set as defined in Definition~\ref{defn:Sing}. Consider  $\phi\colon\GS \to \mathbb{R}$ a H\"older continuous potential. If the pressure of the singular set is strictly less than the full topological pressure, i.e., $P(\Sing,\phi) < P(\phi)$ (see Definitions~\ref{defn: pressure} and ~\ref{defn: Sing pressure}), then $\phi$ has a unique equilibrium state $\mu$ that has the $K$-property (see Definition~\ref{def: K}).
\end{thmintro}

It is natural to ask for which potentials we have the \textit{pressure gap} (i.e., the condition $P(\Sing,\phi)<P(\phi)$) in Theorem~\ref{thm: existence}. The following theorem establishes the pressure gap for a large class of H\"older continuous potentials, and thus uniqueness of equilibrium states.

\begin{thmintro}[Theorem~\ref{thm: locally constant} and Corollary~\ref{cor: nearly constant}]\label{thm: pressure gap} 
Let $\SSS$, $\GS$, and $g_t$ be as in Theorem~\ref{thm: existence}. Let $\phi\colon \GS\rightarrow\mathbb R$ be a H\"older continuous function which is locally constant on a neighborhood of $\Sing$, or which is sufficiently close to a constant in the uniform norm (see Corollary~\ref{cor: nearly constant} for a precise statement of `sufficiently close'). Then $P(\Sing,\phi)<P(\phi)$.
\end{thmintro}

As a nice corollary (Corollary~\ref{cor:top entropy gap} below) we have $h_{top}(g_t|_{\Sing})<h_{top}(g_t)$ for our flows.

We slightly improve the case $\phi=0$ from Ricks's result \cite[Theorem B]{R19} by showing that the unique measure of maximal entropy for the geodesic flow on $\SSS$ has the K-property which is stronger than mixing. Using the Patterson-Sullivan construction, Ricks builds a measure of maximal entropy $\mu$ \cite{R17} and shows it is unique by asymptotic geometry arguments \cite{R19}. We note that Ricks's result holds for any compact, geodesically complete, locally CAT(0) space such that the universal cover admits a rank-one axis.

A natural question is whether the techniques in this paper can be extended to the more general CAT(0), rank-one setting in which Ricks works. The present paper can be viewed as a first step in that direction, but working in the general CAT(0) setting presents real difficulties right from the outset of the argument. In particular, without the Riemannian structure present in \cite{BCFT} or the flat surface structure we exploit, it is not clear to us what the right candidate for the singular set for would be, or how to find a function like $\lambda$ (see Section~\ref{sec:lambda}) to aid in producing an orbit decomposition.

We call a geodesic that is not in $\Sing$ \textit{regular}. Using strong specification for a certain collection of `good' orbit segments, we show that weighted regular closed geodesics equidistribute to these equilibrium states (see Section~\ref{sec: weighted} for details).

\begin{thmintro}[Theorem~\ref{thm: equidistribute}]\label{thm: weighted}
Let $\phi$ be as in Theorem~\ref{thm: pressure gap} and $\mu_\phi$ is the corresponding equilibrium state. Then, $\mu_\phi$ is the weak* limit of weighted regular closed geodesics.
\end{thmintro}

%

\subsection{Outline of the argument}\label{sec:outline}

A general scheme for proving that unique equilibrium states exist was developed by Climenhaga and Thompson in \cite{Climenhaga-Thompson}, building on ideas of Bowen in \cite{bowen} which were extended to flows in \cite{franco}. To prove that there are unique equilibrium states for a flow $\{f_t\}$ and a potential $\phi$ on a compact metric space $X$, Climenhaga and Thompson ask for the following (see \cite[Theorems A \& C]{Climenhaga-Thompson}):
\begin{itemize}
	\item The pressure of obstructions to expansivity, $P^\perp_{exp}(\phi)$ (see Definition~\ref{defn:non-expansive}), is smaller than $P(\phi)$, and
	\item There are three collections of orbit segments $\mathcal{P},\mathcal{G}, \mathcal{S}$, that we call collections of prefixes, good orbit segments, and suffixes, respectively, such that each orbit segment can be decomposed into a prefix, a good part, and a suffix (see \cite[Definition 2.3]{BCFT}), satisfying
	\begin{itemize}
		\item[(I)] $\mathcal{G}$ has the weak specification property at any scale (Definition~\ref{defn:specification}),
		\item[(II)] $\phi$ has the Bowen property on $\mathcal{G}$ (Definition~\ref{def: Bowen property}), and
		\item[(III)] $P([\mathcal{P}] \cup [\mathcal{S}], \phi)<P(\phi)$.
	\end{itemize}
\end{itemize}

This scheme was implemented for the geodesic flow on a closed rank-one manifold with nonpositive sectional curvature in \cite{BCFT} and, more generally, without focal points in \cite{CKP20,CKP}. Also, it was used to obtain the uniqueness of the measure of maximal entropy on certain manifolds without conjugate points in \cite{CKW} and on CAT(-1) spaces in \cite{CLT}.  

Our proof follows a specific approach to satisfying the conditions in the above scheme which was applied in \cite{BCFT}, and which allows us to reduce condition (III) to checking the pressure of an invariant subset of $\GS$. Although the decomposition $(\mathcal{P,G,S})$ is in general very abstract, we choose the decomposition using a function $\lambda$ on the space of geodesics. This choice of decomposition also allows us to avoid having to deal with the sets $[\PPP]$ and $[\SSS]$, which are discretized versions of $\PPP$ and $\SSS$ necessary for technical counting arguments to be applied to some decompositions. We define the function $\lambda$, prove that it is lower semicontinuous, and describe how it gives rise to a decomposition in Section~\ref{sec:lambda}. For such a `$\lambda$-decomposition', $\mathcal{P}=\mathcal{S}$ and, roughly speaking, orbit segments in $\mathcal{P}$ and $\mathcal{S}$ have small average values of $\lambda$ wheareas any initial or terminal segment of an element of $\mathcal{G}$ has average value of $\lambda$ which is not small. Furthermore, by utilizing a $\lambda$-decomposition, we are able to appeal to the following result:

\begin{theorem}[\cite{Ca20}, Theorem 4.6]\label{thm: to get weakly mixing}

Let $\mathcal{F}$ be a continuous flow on a compact metric space $X$, and let $\phi : X\to\mathbb{R}$ be continuous. Suppose the flow is asymptotically entropy-expansive, that $P^\perp_{\exp}(\phi) < P(\phi)$, and that $\lambda : X\to [0,\infty)$ is lower semicontinuous and bounded. If the $\lambda$-decomposition $(\mathcal{P},\mathcal{G},\mathcal{S})$ satisfies the following:
\begin{itemize}
	\item $\mathcal{G}(\eta)$ has strong specification at all scales, for all $\eta > 0$,
	\item $\phi$ has the Bowen property on $\mathcal{G}(\eta)$,
	\item $P(\bigcap_{t\in\mathbb{R}}(f_t\times f_t)\tilde{\lambda}^{-1}(0), \Phi) < 2P(\phi)$,
\end{itemize}
where $\Phi(x,y) = \phi(x) + \phi(y)$ and $\tilde{\lambda}(x,y) = \lambda(x)\lambda(y)$, then $(X,\mathcal{F},\phi)$ has a unique equilibrium state which has the $K$-property.
\end{theorem}

Theorem~\ref{thm: existence} will follow from Theorem~\ref{thm: to get weakly mixing} after we show that we can satisfy all conditions required. See Section~\ref{sec: organization} for the sections where each property is checked. 

Our choice of $\lambda$ gives a connection between orbit segments in $\mathcal{P}$ and $\mathcal{S}$ and the singular set $\Sing$ (see Definition~\ref{defn:Sing}). The singular set is also the source of the obstructions to expansivity (see Lemma~\ref{lemma:NE in Sing}). These connections are useful for proving the two `pressure gap' properties Theorem~\ref{thm: to get weakly mixing} calls for: $P^\perp_{exp}(\phi)<P(\phi)$ and $P(\bigcap_{t\in\mathbb{R}}(f_t\times f_t)\tilde{\lambda}^{-1}(0), \Phi) < 2P(\phi)$. In particular, in our case $\bigcap_{t\in\mathbb{R}}f_t\lambda^{-1}(0)=\Sing$.

\begin{remark*}
The strong specification property on $\mathcal G$ in Theorem~\ref{thm: to get weakly mixing} is used to obtain that the equilibrium state has the $K$-property. The weak specification property on $\mathcal G$ is enough to guarantee the existence of a unique equilibrium state.
\end{remark*}

\begin{remark*}
The K-property implies strong mixing of all orders.
\end{remark*}


\subsection{Organization of the paper}\label{sec: organization}
The paper is organized as follows. 
In Section~\ref{sec:background} we provide definitions of and background on the main objects and tools of this paper and we record some basic geometric results which will be used throughout the paper. The main steps for the proof of Theorem~\ref{thm: existence} according to Theorem~\ref{thm: to get weakly mixing} are in Sections~\ref{sec:lambda} (the $\lambda$-decomposition), \ref{sec:specification} and \ref{sec:strong} (the specification property for $\mathcal{G}$), and \ref{sec:bowen} (the Bowen property for $\mathcal{G}$). 

We obtain Theorem~\ref{thm: pressure gap} in Section~\ref{sec:pressure gap}, first proving the pressure gap condition for potentials which are locally constant on a neighborhood of $\Sing$, and then using this result to note that the same gap holds for potentials with sufficiently small total variation. 
Theorem~\ref{thm: weighted} (the equidistribution result) is proved in Section~\ref{sec: weighted}.

%

\section{Background}\label{sec:background}

\subsection{Setting and Definitions}\label{sec:defns}

Throughout, $\SSS$ denotes a compact, connected surface of genus $\geq 2$ equipped with a metric which is flat everywhere except at finitely many conical points which have angles larger than $2\pi$ (See Figure~\ref{fig:cone point}). We assume $\SSS$ is oriented by passing to the oriented double cover if necessary. $Con$ denotes the set of conical points on $\SSS$ and denote by $\mathcal L(p)$ the total angle at a point $p\in\SSS$.  In particular, $\mathcal{L}(p)=2\pi$ if $p\notin Con$ and $\mathcal{L}(p) >2\pi$ if $p\in Con$. Note that in the special case of a translation surface, $\mathcal{L}(p)$ is always an integer multiple of $2\pi$, but we make no such restriction here. Denote by $\tilde {\SSS}$ the universal cover of $\SSS$, and note that $\tilde\SSS$ is a complete $\CAT(0)$ space (see, e.g. \cite{bh} for definitions and basic results on $\CAT(0)$ spaces). Throughout, tildes denote the obvious lifts to the universal cover.

Since $\tilde \SSS$ is $\CAT(0)$, any $\tilde p, \tilde q$ are connected by a unique geodesic segment. Throughout, we will denote this segment by $[\tilde p, \tilde q].$

\begin{figure}
    \centering
    \includegraphics{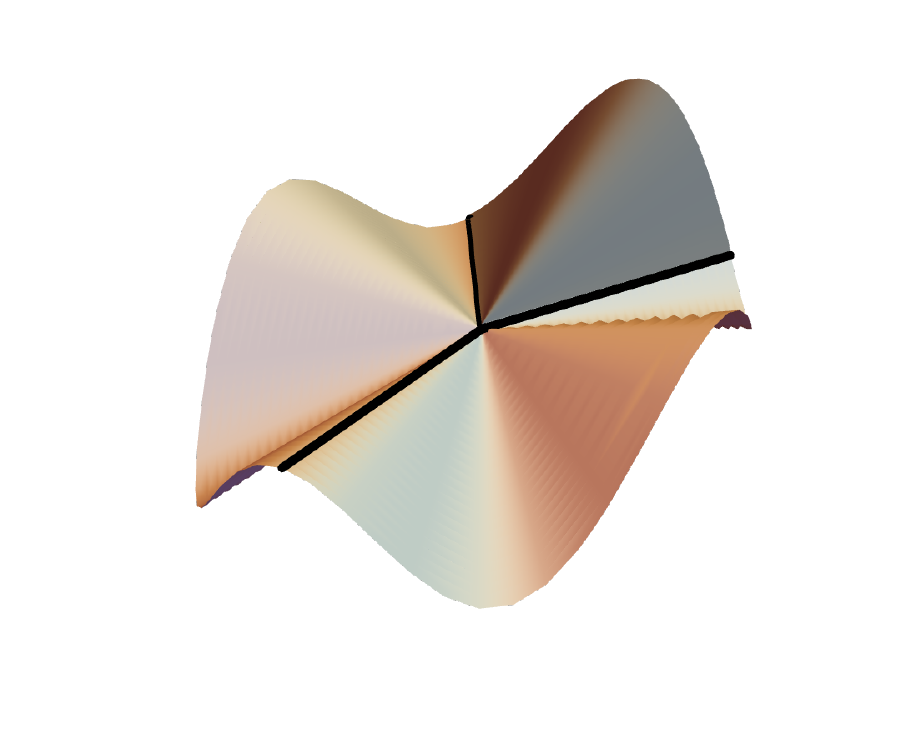}
    \caption{A large-angle cone point, embedded in $\mathbb{R}^3$. Away from the cone point, the surface is flat under the intrinsic metric -- it is the union of lines in $\mathbb{R}^3$ and so has Gaussian curvature zero. The dark lines show a geodesic segment hitting the cone point and its two continuations with turning angles $\pm\pi$; these geodesics are in $\Sing$. All continuations of the geodesic with line segments passing through the dark shaded region are geodesics. The spread of the geodesic continuations in this region is exactly the source of `hyperbolicity' for the geodesic flow in these spaces.}
    \label{fig:cone point}
\end{figure}

Let $\GS$ be the set of all (parametrized) geodesics in $\SSS$. That is, 
\begin{equation}\label{def: GS}
 \GS = \{\gamma\colon \mathbb R\rightarrow \SSS \mid \gamma \text{ is a local isometry}\}.
\end{equation} 
We endow $\GS$ with the following metric:
\begin{equation}\label{metric in GS}
d_{\GS}(\gamma_1,\gamma_2) = \inf\limits_{\tilde\gamma_1,\tilde\gamma_2}\int_{-\infty}^{\infty}d_{\tilde\SSS}(\tilde\gamma_1(t),\tilde\gamma_2(t))e^{-2|t|}\,dt,
\end{equation}
where the infimum is taken over all lifts $\tilde\gamma_i$ of $\gamma_i$ to $\GtS$ for $i=1,2$. $\GS$ serves as an analogue of the unit tangent bundle in our setting. (Indeed, for a Riemannian surface, $\GS$ is homeomorphic to $T^1S$.) It is necessary to examine this more complicated space as geodesics in $\SSS$ are not determined by a tangent vector -- they may branch apart from each other at points in $Con$. In this setting, the metric $d_{\GS}$ records the idea that two geodesics in $\GS$ are close if their images in $\SSS$ are nearby for all $t$ in some large interval $[-T,T]$.

Geodesic flow on $\GS$ comes from shifting the parametrization of a geodesic:
\[ (g_t\gamma)(s) = \gamma(s+t).\]
The normalizing factor 2 in our definition of $d_{\GS}$ ensures that $g_t$ is a unit-speed flow with respect to $d_{\GS}$. (Showing this is a completely straight-forward computation, using the fact that $d_{\tilde\SSS}(\tilde \gamma(t),\tilde\gamma(s+t))=s$).

We recall two definitions of the $K$-property of an invariant measure. See Section 10.8 in \cite{CFSS} for a proof of the equivalence of these definitions (known as completely positive entropy and $K$-mixing, respectively) with the original definition of the $K$-property, as well as more details about other equivalent definitions.

\begin{definition}\label{def: completely positive entropy}
A flow-invariant measure $\mu$ has the $K$-property if $(X,(g_t),\mu)$ has no non-trivial zero entropy factors (i.e., the Pinsker factor is trivial).
\end{definition}

This definition can be reformulated as a statement about mixing in the following manner.

\begin{definition}\label{def: K}
A flow-invariant measure $\mu$ has the $K$-property if for all $t\neq 0$, for all $k\geq 1$, and all measurable sets $A_0,A_1,\ldots, A_k$ we have 
\begin{equation*}
    \lim\limits_{n\rightarrow\infty}\sup\limits_{B\in\mathcal C_n(A_1,\ldots,A_k)}\left|\mu(A_0\cap B)-\mu(A_0)\mu(B)\right|=0,
\end{equation*}
where $\mathcal C_n(A_1,\ldots,A_k)$ is the minimal $\sigma$-algebra generated by $g_{tr}(A_j)$ for $1\leq j\leq k$ and natural $r\geq n$.
\end{definition}

\begin{remark*}
The K-property implies strong mixing of all orders. We recall that an invariant measure $\mu$ is strongly mixing of all orders if for all $k\geq 1$ and all measurable sets $A_0,A_1,\ldots, A_k$ we have 
\begin{equation*}
    \lim\limits_{t_1\rightarrow\infty,\, t_{j+1}-t_j\rightarrow\infty}\mu(A_0\cap g_{t_1}(A_1)\cap\ldots\cap g_{t_k}(A_k))=\prod_{j=0}^{k}\mu(A_j).
\end{equation*}
\end{remark*}

A key tool in our analysis of the geodesic flow on $S$ will be the turning angle of a geodesic at a cone point.  We note that although $\SSS$ is not smooth at $p\in Con$, there is a well-defined space of directions at $p$, $S_p\SSS$, and a well-defined notion of angle (see, e.g. \cite[Ch. II.3]{bh}). In the angular metric, $S_p\SSS$ is a circle of total circumference $\mathcal{L}(p)$.

\begin{definition}\label{defn:turning angle}
	Let $\gamma\in \GS$. The turning angle of $\gamma$ at time $t$ is $\theta(\gamma,t)\in (-\frac{1}{2}\mathcal L(\gamma(t)),\frac{1}{2}\mathcal L(\gamma(t))]$ and is the signed angle between the segments $[\gamma(t-\delta),\gamma(t)]$ and $[\gamma(t),\gamma(t+\delta)]$ (for sufficiently small $\delta>0$). A positive (resp. negative) sign for $\theta$ corresponds to a  counterclockwise (resp. clockwise) rotation with respect to the orientation of  $[\gamma(t-\delta),\gamma(t)]$.
\end{definition}

Since $\gamma$ is a geodesic, $|\theta(\gamma, t)|-\pi\geq 0$ for any $t\in\mathbb R$. If $\gamma(t)\not\in Con$, then $\theta(\gamma, t)=\pi$.

\begin{definition}\label{defn:Sing}
We define the singular geodesics in $\SSS$ as
\begin{equation*}
\Sing =\{\gamma \in \GS : |\theta(\tilde\gamma,t)|=\pi \quad \forall t\in\mathbb R\}.
\end{equation*}
 Since $\Sing$ is defined in terms of properties of full geodesics, it is $g_t$-invariant. Geodesics not in $\Sing$ turn by some angle $\neq \pi$ at a cone point. This is an open condition, so $\Sing$ is closed and hence compact.
\end{definition}

The geodesics in $\Sing$ either never encounter any cone points or, when they do, turn by angle exactly $\pm \pi$. They serve as an analogue of the singular set in the Riemannian setting of \cite{BCFT}, i.e., geodesics which remain entirely in zero-curvature regions of the surface. In both cases the idea is that a singular geodesic never takes advantage of the geometric features of the surface (either its negative curvature regions or its large-angle cone points) to produce hyperbolic dynamical behavior.  We note here a potentially confusing aspect of this terminology: a singular geodesic in this paper avoids the `singular,' i.e. non-smooth, points of $Con$, or treats them as if they are not `singular.'

We introduce some classical notions of thermodynamical formalism.

\begin{definition}\label{defn: pressure}
Consider a function $\phi\colon \GS\rightarrow\mathbb R$ that we refer as a \textit{potential function}. The \textit{pressure} for $\phi$ is
\begin{equation*}
P(\phi) = \sup\limits_{\mu}\left(h_{\mu}(g_t)+\int_{\GS}\phi d\mu\right),
\end{equation*}
where $\mu$ varies over all invariant Borel probability measures for $g_t$ and $h_{\mu}(g_t)$ is the measure-theoretic entropy with respect to the geodesic flow.

An invariant Borel probability measure $\mu_\phi$ (if it exists) such that
\begin{equation*}
    P(\phi)=h_{\mu_\phi}(g_t)+\int_{\GS}\phi d\mu_\phi
\end{equation*}
is an \textit{equilibrium state} for $\phi$.
\end{definition}

\begin{definition}\label{defn: Sing pressure}
$P(\Sing, \phi)$ is the pressure of the potential $\phi|_{\Sing}$ on the compact and flow-invariant set $\Sing$ (see Definition~\ref{defn:Sing}).
\end{definition}

Below, we discuss some of the necessary definitions to apply the Climenhaga-Thompson machinery.

\begin{definition}\label{defn:non-expansive}
	Let $\eps>0$. The non-expansive set at scale $\eps$ for the flow $g_t$ is 
	\[ \NE(\eps) = \{\gamma \in \GS \mid \Gamma_{\eps}(\gamma)\not\subset g_{[-s,s]}\gamma \text{ for all } s > 0 \}, \]	
	where
	\[ \Gamma_{\eps}(\gamma) = \{\xi \in \GS \mid d_{\GS}(g_t\gamma,g_t\xi)\leq\eps \quad \forall t\in\mathbb R\}. \]
	The pressure of obstructions to expansivity for a potential $\phi$ is
	\begin{equation*}
	P^\perp_{\exp}(\phi) = \lim\limits_{\eps\downarrow 0}\sup\left\{h_{\mu}(g_1)+\int_{\GS}\phi\,d\mu \bigm| \mu (\NE(\eps)) = 1\right\},
	\end{equation*}
	where the supremum is taken over all $g_t$-invariant ergodic probability measures $\mu$ on $\GS$ such that $\mu(\NE(\eps))=1$.
\end{definition}

In other words, a geodesic is in the complement of $\NE(\eps)$ if the only geodesics which stay $\eps$ close to it for all time are contained in its own orbit. A flow is expansive if $\NE(\eps)$ is empty for all sufficiently small $\eps$. The presence of flat strips in our setting means our flow will not be expansive, but for small $\eps$, the complement of $\NE(\eps)$ will turn out to be a sufficiently rich set to use in our arguments.

In the interest of concision, we omit the formal definition of an orbit decomposition, referring instead to \cite{Climenhaga-Thompson}. We will use a specific type of decomposition which has been studied in \cite{CT19,Ca20}, and we will primarily use results from those two papers. We note however that results from \cite{Climenhaga-Thompson} hold for our decompositions as well, as it is written for a more general class of decomposition. We discuss this more in Section~\ref{sec: weighted}, where we will need to appeal to a few results directly from \cite{Climenhaga-Thompson}. Identify a pair $(\gamma,t)\in \GS\times [0,\infty)$ with the \emph{orbit segment} $\{g_s\gamma \mid s\in [0,t]\}$. An orbit decomposition is a method of decomposing any orbit segment into three subsegments, a prefix, a central good segment, and a suffix. We denote the collections of these segments by $\PPP,\mathcal{G}$, and $\mathcal{S}$ respectively. The $\lambda$-decompositions that we use in this paper are orbit decompositions which decompose orbit segments based on a lower semicontinuous function $\lambda$. Our choices for the function $\lambda$ and the associated parameter $\eta>0$ will be discussed in detail in Section \ref{sec:lambda}, but the idea is this. The function $\lambda$ measures the amount of `hyperbolic' behavior seen by the geodesic; in accord with our intuition that cone points are the source of this behavior, $\lambda$ will be based on turning angles at these points. A segment is `good' for our purposes (i.e., in $\GGG(\eta)$) if it experiences a lot of hyperbolicity; otherwise, it is in $\mathcal{P}=\mathcal{S}$: 

\begin{itemize}
    \item $\GGG=\GGG(\proportion)$ consists of all $(\gamma,t)$ such that the average value of $\lambda$ over every initial and terminal segment of $(\gamma,t)$ is at least $\proportion$, and
    \item $\mathcal{P}=\mathcal{S}=\mathcal{B}(\proportion)$ consists of all $(\gamma,t)$ over which the average value of $\lambda$ is less than $\proportion.$
\end{itemize}

We can define both specification and the Bowen property for an arbitrary collection of orbit segments $\GGG \subset \GS\times [0,\infty)$. In both cases, by taking $\GGG = \GS\times [0,\infty)$, one retrieves the definitions for the full dynamical system.
\begin{definition}\label{defn:specification}
	We say that $\GGG$ has \textit{weak specification} if for all $\eps > 0$, there exists $\tau > 0$ such that for any finite collection $\{(x_i,t_i)\}_{i=1}^n\subset \GGG$, there exists $y\in \GS$ that $\eps$-shadows the collection with transition times $\{\tau_i\}_{i=1}^n$ at most $\tau$ between orbit segments. In other words, for $1\leq i \leq n$, there exists $\tau_i \in [0,\tau]$ and $y\in \GS$ such that 
	$$d_{\GS}(g_{t + s_i} y, g_t x_i) \leq \eps \text{ for } 0\leq t \leq t_i$$
	where $s_k = \sum_{j=1}^{k-1} t_j + \tau_j$. We will refer to such $\tau$ as a \textit{specification constant}.

	We say that $\GGG$ has \textit{strong specification} when we can always take each $\tau_j = \tau$ in the above definition.
	
\end{definition}

\begin{definition}\label{def: Bowen property}
	Given a potential $\phi : \GS\to\mathbb{R}$, we say that $\phi$ has the Bowen property on $\GGG$ if there is some $\eps > 0$ for which there exists a constant $K > 0$ such that
	\begin{equation*}
	\sup\left\{\left|\int_0^t \phi(g_rx) - \phi(g_ry)\,dr\right| \bigm| (x,t)\in \GGG \text{ and } d_{\GS}(g_ry,g_rx) \leq \eps \text{ for } 0\leq r \leq t\right\} \leq K.
	\end{equation*}
\end{definition}

\begin{remark*}
If $\phi$ has the Bowen property on a collection of orbit segments $\GGG$ at some scale $\eps > 0$, it in turn has the Bowen property on $\GGG$ at all smaller scales $\eps' < \eps$.
\end{remark*}

There is also a definition of topological pressure for collections of orbit segments. However, by using Theorem~\ref{thm: to get weakly mixing}, we sidestep this complication.

Finally, we adapt a piece of terminology from flat surfaces to our somewhat more general setting.

\begin{definition}\label{defn:saddle connection}
A geodesic segment with both endpoints in $Con$ and no cone points in its interior is called a saddle connection. A saddle connection path is composed of saddle connections joined so that the turning angle at each cone point is at least $\pi$. Note that with this definition all saddle connection paths are geodesic segments.
\end{definition}

%

\subsection{Basic geometric results}\label{sec:basic geom}

In this section we collect a few basic results on the geometry of $\SSS$, $\tilde\SSS$, $\GS$, and $\GtS$ which will be used in our subsequent arguments.

The following two lemmas relate the metric $d_{\GS}$ to the metric $d_\SSS$ on the surface itself, and will be useful for a number of our calculations below. First, we note that if two geodesics are close in $\GS$, then they are close in $\SSS$ at time zero.

\begin{lemma}[\cite{CLT}, Lemma 2.8]\label{lem:closeness in S and GS}
For all $\gamma_1, \gamma_2\in \GS$, 
\[ d_\SSS(\gamma_1(0), \gamma_2(0)) \leq 2d_{\GS}(\gamma_1,\gamma_2).\]
Furthermore, for $s, t \in \mathbb R$, $d_\SSS(\gamma_1(s),\gamma_2(t)) \leq 2 d_{\GS}(g_s\gamma_1,g_t \gamma_2).$
\end{lemma}

Conversely, if two geodesics are close in $\SSS$ for a significant interval of time surrounding zero, then they are close in $\GS$:

\begin{lemma}[\cite{CLT}, Lemma 2.11]\label{lem:shadow in S}
Let $\eps$ be given and $a<b$ arbitrary. There exists $T=T(\eps)>0$ such that if $d_\SSS(\gamma_1(t),\gamma_2(t))<\eps/2$ for all $t\in[a-T,b+T]$, then $d_{\GS}(g_t\gamma_1,g_t\gamma_2)<\eps$ for all $t\in[a,b]$. For small $\eps$, we can take $T(\eps)= -  \log(\eps)$.
\end{lemma}

A similar, and more specialized result which we will need later in the paper (see the proof of Proposition~\ref{prop: Bowen property}) is the following

\begin{lemma}\label{lem:exponentially close}
Suppose that $d_\SSS(\gamma_1(t),\gamma_2(t))=0$ for all $t\in[a,b]$. Then, for all $t\in[a,b]$, $d_{\GS}(g_t\gamma_1, g_t\gamma_2)\leq e^{-2\min\{ |t-a|,|t-b| \}}$.
\end{lemma}

\begin{proof}
For any $x\geq 0$, $\int_x^\infty (s-x)e^{-2s}ds = \frac{1}{4}e^{-2x}$. In the setting of the Lemma, since the distance between the geodesics is zero on $[a,b]$ and since geodesics move at unit speed,
\[ d_{\GS}(g_t\gamma_1,g_t\gamma_2) \leq \int_{-\infty}^a 2(a-s)e^{-2|t-s|}ds + \int_b^\infty 2(s-b)e^{-2|t-s|}ds.\] 
Quick changes of variables show that this is equal to $\int_{|t-a|}^\infty2(s-|t-a|)e^{-2s}ds+\int_{|t-b|}^\infty2(s-|t-b|)e^{-2s}ds=\frac{1}{2}(e^{-2|t-a|}+e^{-2|t-b|})$, and the Lemma follows.
\end{proof}

The geodesic flow has the following Lipschitz property:

\begin{lemma}[\cite{CLT2}, Lemma 2.5]\label{lem:Lipschitz}
Fix a $T>0$. Then, for any $t\in[0,T]$, and any pair of geodesics $\gamma,\xi \in \GS$,
\[ d_{\GS}(g_t \gamma, g_t \xi) < e^{2T} d_{\GS}(\gamma,\xi). \]
\end{lemma}

We need the following four geometric facts.

\begin{lemma}\label{lem:no big flat}
\begin{itemize}
	\item[(a)] There exists some $d_0>0$ such that $\tilde \SSS$ contains no flat $d_0 \times d_0$ square.
	\item[(b)] There exists some $\eta_0>0$ such that the excess angle at every cone point in $\SSS$ is at least $\eta_0$.
	\item[(c)] There exists some $\ell_0>0$ such that the length of every saddle connection is at least $\ell_0$.
	\item[(d)] There exists some $\theta_0>0$ such that the excess angle at every cone point in $\SSS$ is at most $\theta_0$.
\end{itemize}
\end{lemma}

\begin{proof}
These follow immediately from the compactness of $\SSS$ and the fact that $\SSS$ having genus at least two implies $Con \neq \emptyset.$
\end{proof}

We note here that $\Sing$ is the source of the non-expansivity for our geodesic flow:

\begin{lemma}\label{lemma:NE in Sing}
For all $\eps>0$ less than half the injectivity radius of $\SSS$, $\NE(\eps)\subset \Sing$.
\end{lemma}

\begin{proof}
Suppose $\gamma \in \NE(\eps)$ and that $\eps$ is smaller than half the injectivity radius of $\SSS$. Then, there exists $\xi \in \GS$ which is not in the orbit of $\gamma$ such that $d_{\GS}(g_t\gamma,g_t\xi)\leq \eps$ for all $t\in\mathbb R$. By Lemma~\ref{lem:closeness in S and GS}, $d_{\SSS}(\gamma(t),\xi(t))\leq 2\eps$ for all $t\in\mathbb R$. In particular, using our assumption on $\eps$, there exist lifts $\tilde\gamma$ and $\tilde\xi$ such that $d_{\tilde\SSS}(\tilde\gamma(t),\tilde\xi(t))\leq 2\eps$ for all $t\in\mathbb R$. By the Flat Strip Theorem (\cite[Corollary 5.8 (ii)]{Ballmann}), there is an isometric embedding $\mathbb{R}\times [a,b] \to \tilde\SSS$ sending $\mathbb{R}\times \{a\}$ to the image of $\tilde \gamma$ and $\mathbb{R}\times \{b\}$ to the image of $\tilde \xi$. 

Since $\tilde \xi$ is not in the orbit of $\tilde \gamma$, we must have $a\neq b$ and the isometrically embedded strip is non-degenerate. But this immediately implies that for all $t$, $|\theta(\tilde \gamma, t)|=\pi$ as $\tilde\gamma$ always turns at angle $\pi$ on the side to which the embedded flat strip lies. Therefore, $\gamma \in \Sing$.
\end{proof}

Recall that a flow is called entropy-expansive if for sufficiently small $\eps$, $\sup \{ h_{top}(g_t|_{\Gamma_\eps(\gamma)}) \mid \gamma\in \GS \} =0$.

\begin{lemma}[\cite{R19}, Lemma 20]\label{lem:entropy-expansive}
The geodesic flow in our setting is entropy-expansive.
\end{lemma}

\begin{proof}
This is proven by Ricks in \cite{R19} for geodesic flow on a CAT(0) space. This covers our setting, but Ricks uses a slightly different definition of the metric on $\GS$ than we do, so we outline the argument here.

Fix $\eps$ less than half the injectivity radius of $\SSS$. Lift $\gamma$ to $\tilde\gamma\in G\tilde \SSS$.  Any geodesics $\xi\in \Gamma_\eps(\gamma)$ lift to $\tilde \xi \in \Gamma_\eps(\tilde \gamma)$. They are either of the form $g_t\tilde \gamma$ for $|t|<\eps$, or are parallel to $\tilde\gamma$ in a flat strip containing $\tilde\gamma$. The flow on $\Gamma_\eps(\tilde \gamma)$ is thus isometric, and so $h_{top}(g_t|_{\Gamma_\eps(\gamma)})=0$.
\end{proof}

\begin{lemma}\label{lem:closed saddle}
Given any closed geodesic $\gamma \subset \SSS$, there is a closed saddle connection path which is homotopic to $\gamma$ and has the same length as $\gamma$.
\end{lemma}

\begin{proof}
Assume $\gamma$ contains a point $p\in Con$. Then the desired closed saddle connection path is the geodesic that starts at $p$ and traces $\gamma$.

Suppose $\gamma \subset \SSS \setminus Con$, and so $\tilde \gamma \subset \tilde\SSS \setminus \widetilde{Con}$. Fix an orientation of $\tilde \gamma$ and consider the variation $\tilde \gamma_r$ of curves given by sliding $\tilde \gamma$ to its left (so the variational field is perpendicular to $\tilde \gamma$ and to its left with respect to $\tilde\gamma$'s orientation). Since $\tilde \gamma \subset \tilde\SSS \setminus \widetilde{Con}$ and $\gamma$ is closed, there is a nonzero lower bound on the distance from $\tilde \gamma$ to $\widetilde{Con}$. Therefore, for all sufficiently small $r$, $\tilde \gamma_r$ is defined. The projections to $\SSS$, $\gamma_r$ and $\gamma$, form the boundary of a flat cylinder in $\SSS$. Thus, $\gamma_r$ is a geodesic with length equal to that of $\gamma$.

Let $r^*$ be the supremum of all $r>0$ for which $\tilde \gamma_\rho$ is defined for all $\rho\in[0,r]$. Note that if no supremum exists, $\tilde \gamma$ bounds a flat half-space in $\tilde \SSS$, which contains a fundamental domain for $\SSS$ since $\SSS$ is compact. This would imply $\SSS$ is flat (with \emph{no} cone points), a contradiction. Therefore letting $r\to r^*$ from below, $\tilde\gamma_r$ limits uniformly on a path, and therefore necessarily a geodesic, containing at least one point in $\widetilde{Con}$ with the same length as $\tilde\gamma$. The image of this curve in $\SSS$ (with appropriate parametrization) is the saddle connection path we want.
\end{proof}

In the proof of Lemma~\ref{existence of a turning closed geodesic} and in some later proofs we will use the following construction.

\begin{definition}\label{defn:cone}
Let $\tilde \gamma$ be a geodesic segment in $\tilde \SSS$ with endpoint $p$. The cone around $\tilde \gamma$ with vertex $p$ and angle $\psi$ is the set of all points $q$ in $\tilde \SSS$ such that the unique geodesic segment joining $p$ and $q$ makes angle $\leq \psi$ with $\tilde \gamma$ at $p$. (In Section~\ref{sec:lambda}, Figure~\ref{fig:case3} shows such cones in the context of the proof of Lemma~\ref{lem:lower semicts}.)
\end{definition}

\begin{lemma}\label{existence of a turning closed geodesic}
For any $\zeta \in Con$ there exists a closed geodesic $\alpha$ passing through $\zeta$ with turning angle greater than $\pi$ at $\zeta$.
\end{lemma}

\begin{proof}
Let $\zeta$ be a cone point with $\mathcal{L}(\zeta)=2\pi+\beta$ for $\beta>0$.  Lift $\zeta$ to $\tilde \zeta$ in $\tilde \SSS$ and let $\tilde c$ be a geodesic with $\tilde c(0)=\tilde\zeta$ and turning angle $\theta(\tilde c,0)=\pi+\frac{\beta}{2}$. Let $C_1$ be the cone around $\tilde c(-\infty,0)$ with vertex $\tilde \zeta=\tilde c(0)$ and angle $\frac{\beta}{8}$; let $C_2$ be the cone around $\tilde c(0,\infty)$ with vertex $\tilde \zeta=\tilde c(0)$ and angle $\frac{\beta}{8}$. By construction, any geodesic connecting a point in $C_1\setminus\{ \tilde \zeta\}$ to a point in $C_2\setminus\{ \tilde \zeta\}$ must pass through $\tilde \zeta$ with turning angle $\geq \pi+\frac{\beta}{4}$. 

Let $\mathcal{F}$ be a fundamental domain contained in $C_1$. Let $g\in \pi_1(\SSS)$ be such that $g\mathcal{F} \subset C_2$. ($\mathcal{F}$ and $g$ exist as both $C_1$ and $C_2$ contain arbitrarily large balls and $\SSS$ is compact.) Let $\alpha$ be the closed geodesic representative of $g$ in $\SSS$. (It will become clear in a moment why $\alpha$ is unique up to parametrization.) Lift $\alpha$ to $\tilde \alpha$ with $\tilde \alpha(0)\in \mathcal{F}$. Then $\tilde \alpha(\ell(\alpha)) \in g\mathcal{F}$. As noted above, this forces $\tilde\alpha$ to pass through $\tilde \zeta$ and turn with angle $>\pi$. Therefore, $\alpha$ is the desired geodesic (and it is unique up to parametrization since it cannot belong to a flat cylinder).
\end{proof}

%

\section{The $\lambda$-decomposition}\label{sec:lambda}

We now turn to the main arguments of the paper. First, following the ideas in \cite{BCFT}, we establish the decomposition $(\mathcal{P,G,S})$ as a `$\lambda$-decomposition' using the function $\lambda$ in Definition~\ref{defn:lambda} which is defined through two auxiliary functions that view the stable and unstable parts of any given geodesic. 
Throughout this section, fix $s>0$ such that $2s$ is less than the shortest saddle connection of $\SSS$. Below we omit in the notation the dependence of functions on $s$.

\begin{definition}\label{defn: lambda unstable}
    We define $\lambda^{uu}\colon \GS\rightarrow [0,\infty)$ by
    \begin{equation*}
        \lambda^{uu}(\gamma) = \frac{|\theta(\gamma,c)| - \pi}{\max\{s,c\}},
    \end{equation*}
    where $c \geq 0$ is the first time that $\gamma(c)$ hits a cone point and turns with angle strictly greater than $\pi$ (naturally, we set  $\lambda^{uu}(\gamma)=0$ in case $c=\infty$).
\end{definition}

\begin{definition}\label{defn: lambda stable}
	We define $\lambda^{ss}\colon \GS\rightarrow [0,\infty)$ by
    \begin{equation*}
        \lambda^{ss}(\gamma) = \frac{|\theta(\gamma,c)| - \pi}{\max\{s,|c|\}},
    \end{equation*}
    where $c \leq 0$ is the most recent time that $\gamma(c)$ has hit a cone point and turned with angle strictly greater than $\pi$ (naturally, we set  $\lambda^{ss}(\gamma)=0$ in case $c=-\infty$).
\end{definition}

We now define our function $\lambda$ so that near cone points at which geodesics turn with angle greater than $\pi$, it measures the turning angle at that cone point (multiplied by a constant), and far from a cone point, it measures both distance and turning angle from both the previous and next cone point.

\begin{definition}\label{defn:lambda}
Let $\lambda^{uu}$ and $\lambda^{ss}$ be functions defined in Definitions~\ref{defn: lambda unstable} and ~\ref{defn: lambda stable}, respectively. We define $\lambda\colon \GS\rightarrow [0,\infty)$ by
\begin{equation*}
\lambda(\gamma) = \begin{cases}
\lambda^{ss}(\gamma) &\text{if there exists }c\in (-s,0] \text{ such that } |\theta(\gamma,c)| - \pi > 0,\\
\lambda^{uu}(\gamma) &\text{if there exists }c\in [0,s) \text{ such that } |\theta(\gamma,c)| - \pi > 0,\\
\min\{\lambda^{ss}(\gamma),\lambda^{uu}(\gamma)\} &\text{otherwise}.
\end{cases}
\end{equation*}
Observe that it is well-defined when $\gamma(0)$ is a cone point, as in that case, $\lambda^{uu}(\gamma) =\lambda^{ss}(\gamma)$.
\end{definition}

We prove several properties of $\lambda$.

\begin{proposition}\label{nonrecurrent}
	If $\lambda(\gamma) = 0$, then $\lambda(g_t\gamma) = 0$ either for all $t \geq 0$ or for all $t \leq 0$.
\end{proposition}

\begin{proof}
	If $\lambda(\gamma) = 0$, then $\gamma$ does not turn at a cone point in the interval $(-s,s)$, and so, $\lambda^{uu}(\gamma) = 0$ or $\lambda^{ss}(\gamma) = 0$. In the first case, this implies that $\gamma$ never turns at a cone point in the future. Therefore, for all $t \geq 0$, $\lambda(g_t\gamma) = \lambda^{uu}(g_t\gamma) = 0$. A similar argument holds with $t\leq 0$ if $\lambda^{ss}(\gamma) = 0$.
\end{proof}

As corollaries, we have:

\begin{corollary}
$\bigcap_{t\in\mathbb{R}}g_t\lambda^{-1}(0) = \Sing$.
\end{corollary}

\begin{corollary}\label{cor: converges to Sing}
If $\lambda(\gamma) = 0$, then $d(g_t\gamma,\Sing) \to 0$ either as $t\to\infty$ or as $t\to-\infty$.
\end{corollary}

\begin{proof}
Without loss of generality, assume $\lambda(g_t\gamma) = 0$ for all $t\geq 0$. Then, $\gamma$ does not turn at a cone point in $[0,\infty)$, and we can define $r := \max\{t : |\theta(\gamma,t)| > \pi\}$ to be the most recent cone point in the past at which $\gamma$ turns. Define a singular geodesic $\gamma_\Sing$ as $\gamma_\Sing(t) = \gamma(t)$ for all $t > r$, and for all cone points $t\leq r$, $\gamma_\Sing$ turns with angle $\pi$. Then, $g_t\gamma$ and $g_t\gamma_\Sing$ agree on increasingly long intervals, and by Lemma~\ref{lem:exponentially close} for $t > r$, 
$$d_{\GS}(g_t\gamma,\Sing) \leq d_{\GS}(g_t\gamma,g_t\gamma_\Sing) \leq e^{-2(t-r)}.$$
The proof if $\lambda(g_t\gamma) = 0$ holds similarly, but sending $t\to -\infty$ instead.
\end{proof}

Furthermore, this allows us to show that the pressure gap for the product flow (condition (3) of Theorem~\ref{thm: to get weakly mixing}) is implied by the pressure gap $P(\Sing,\phi) < P(\phi)$ that we will establish in Section~\ref{sec:pressure gap}.

\begin{proposition}(Following \cite[Proposition 5.1]{CT19})
Setting $\Phi(x,y) = \phi(x) + \phi(y)$ and $\tilde{\lambda}(x,y) = \lambda(x)\lambda(y)$ we have $P(\bigcap_{t\in\mathbb{R}}(g_t\times g_t)(\tilde{\lambda})^{-1}(0),\Phi) \leq P(\phi) + P(\Sing,\phi)$. In particular, if $P(\Sing,\phi) < P(\phi)$, then $P(\bigcap_{t\in\mathbb{R}}(g_t\times g_t)(\tilde{\lambda})^{-1}(0),\Phi) < 2P(\phi)$.
\end{proposition}

\begin{proof}
The Variational Principle \cite[Theorem 9.10]{Wa} tells us that 
$$P\left(\bigcap_{t\in\mathbb{R}}(g_t\times g_t)(\tilde{\lambda}^{-1}),\Phi\right) = \sup\left\{P_\nu(\Phi) \bigm| \nu \text{ is flow invariant and }\nu\left(\bigcap_{t\in\mathbb{R}}(g_t\times g_t)(\tilde{\lambda}^{-1})\right) = 1\right\},$$
where $P_\nu(\Phi) := h_\nu(g_1\times g_1) + \int\Phi\,d\nu$ denotes the measure-theoretic pressure of $(\bigcap_{t\in\mathbb{R}}(g_t\times g_t)(\tilde{\lambda}^{-1}),(g_t\times g_t),\Phi,\nu).$ More generally, this relationship holds for any continuous flow, continuous potential, and compact, flow-invariant subset.

Consequently, we let $\nu$ be an invariant measure supported on $\bigcap_{t\in\mathbb{R}} (g_t\times g_t)(\tilde{\lambda})^{-1}(0)$, and let
$$A = \bigcap_{t\in\mathbb{R}} (g_t\times g_t)(\tilde{\lambda})^{-1}(0) \cap (\Reg\times \Reg).$$
We will show that $\nu(A) = 0$ by showing that it contains no recurrent points. Assume for contradiction that $(\gamma_1,\gamma_2)\in A$ is a recurrent point, and then assume without loss of generality that $\lambda(\gamma_1) = 0$. Since $\gamma_1\notin \Sing$, it follows that $d_{G\SSS}(\gamma_1,\Sing) = c > 0$, which from recurrence, implies that there exists a sequence $t_k\to\infty$ such that $d_{G\SSS}(g_{t_k}\gamma_1,\Sing) > \frac{c}{2}$, with a similar claim holding in backwards time. However, we also know that $d_{G\SSS}(g_t\gamma_1,\Sing) \to 0$ as $t\to\infty$, or as $t\to-\infty$ by Corollary~\ref{cor: converges to Sing}. Thus, we have arrived at a contradiction. Hence, $\nu$ is supported on the complement of $\Reg\times\Reg$, which is $(\Sing\times \GS) \cup (\GS\times \Sing)$.

Thus,
$$P\left(\bigcap_{t\in\mathbb{R}}(g_t\times g_t)(\tilde{\lambda}^{-1}(0)),\Phi\right) \leq P\left((\Sing\times \GS)\cup (\GS\times \Sing),\Phi\right)\leq P(\Sing,\phi) + P(\GS,\phi).$$
The first inequality is by the Variational Principle. The second inequality is due to the fact that the pressure of the union of two compact invariant sets is the maximum of the pressure of each individual set \cite[Theorem 11.2(3)]{PesinDimensionTheory}, and in this case, the pressure of each component of the union is at most $P(\Sing,\phi) + P(\GS,\phi)$ by \cite[Theorem 9.8(v)]{Wa}.
\end{proof}

We have also constructed $\lambda$ so that it is lower semicontinuous.

\begin{lemma}\label{lem:lower semicts}
Let $s>0$ be such that $2s$ is less than the shortest saddle connection of $\SSS$. Then, $\lambda$ defined in Definition~\ref{defn:lambda} is lower semicontinuous.
\end{lemma}

\begin{proof}
Let $\gamma\in\GS$. We show that for any $\eps>0$ there exists $\delta>0$ such that $\lambda(\gamma)-\eps<\lambda(\xi)$ for all $\xi\in\GS$ such that $d_{\GS}(\gamma,\xi)<\delta$. To ease the arguments below slightly, we work in $\tilde\SSS$ with lifts $\tilde \gamma, \tilde \xi$ so that $d_{\GtS}(\tilde\gamma, \tilde \xi) = d_{\GS}(\gamma, \xi)$. Recall that by Lemma~\ref{lem:closeness in S and GS}, if $d_{\GtS}(\tilde \gamma,\tilde \xi)<\delta$ then $d_{\tilde\SSS}(\tilde \gamma(0),\tilde \xi(0))<2\delta$. 

If  $\lambda(\gamma)=0$, then we are done as $\lambda$ is a non-negative function. Therefore, for the rest of the argument we assume that $\lambda(\gamma)>0$.

\underline{Case 1}: Suppose there exists $c\in (-s,s)$ such that $\psi:=|\theta(\tilde\gamma, c)|-\pi>0$. Denote $\tilde\gamma(c)=p$. We show that there exists $\delta>0$ such that $p\in\tilde\xi((-s,s))$. 

Let $C_1$ be the cone around $\tilde\gamma((c,s))$ with vertex $p$ and angle $\psi'=\min\{\frac{\psi}{2},\frac{\pi}{4}\}$. Let $C_2$ be the cone around $\tilde\gamma((-s,c))$ with vertex $p$ and angle $\psi'$. (See Figure~\ref{fig:case3}.)
Set
\begin{equation}\label{choice delta}
\delta=\frac{1}{2}\min\left\{\frac{1}{8}e^{-2s}(s-|c|)\sin\psi',\quad \frac{1}{2}(s-d_{\SSS}(\tilde\gamma(0),p)),\quad \min\{1,\eps s/8\}(2+e^{2|c|})^{-1}\int_0^{s-|c|}2te^{-2t}dt\right\}.
\end{equation}
Then, we choose $u_1=\frac{c+s}{2},u_2=\frac{c-s}{2}$ and $\delta_1=\frac{s-c}{4}\sin\psi',\delta_2=\frac{c+s}{4}\sin\psi'>0$ so that $B_1:=B(\tilde\gamma(u_1),\delta_1)\subset C_1\cap B(\tilde\gamma(0),s)$ and $B_2:=B(\tilde\gamma(u_2),\delta_2)\subset C_2\cap B(\tilde\gamma(0),s)$. By Lemmas~\ref{lem:closeness in S and GS} and \ref{lem:Lipschitz}, since $\delta<\frac{1}{8}e^{-2s}(s-|c|)\sin\psi'$, if $d_{\GtS}(\tilde\gamma,\tilde\xi)<\delta$ then $\tilde\xi$ passes through $B_1$ and $B_2$. Since any two points in a CAT(0)-space are connected by a unique geodesic segment and by our construction of $C_1$ and $C_2$, we obtain that if $d_{\GtS}(\tilde\gamma,\tilde\xi)<\delta$ then $\tilde\xi$ passes through $p$. Furthermore, since $2\delta+d_{\SSS}(\tilde\gamma(0),p)<s$, by Lemma~\ref{lem:closeness in S and GS} and the triangle inequality for the triangle with vertices $\tilde\xi(0), \tilde\gamma(0)$, and $p$, we have $p\in\tilde\xi((-s,s))$ if $d_{\GtS}(\tilde\gamma,\tilde\xi)<\delta$. Let $t_0\in(-s,s)$ be such that $\tilde\xi(t_0)=p$. Moreover, $|t_0-c|\leq 2\delta$.

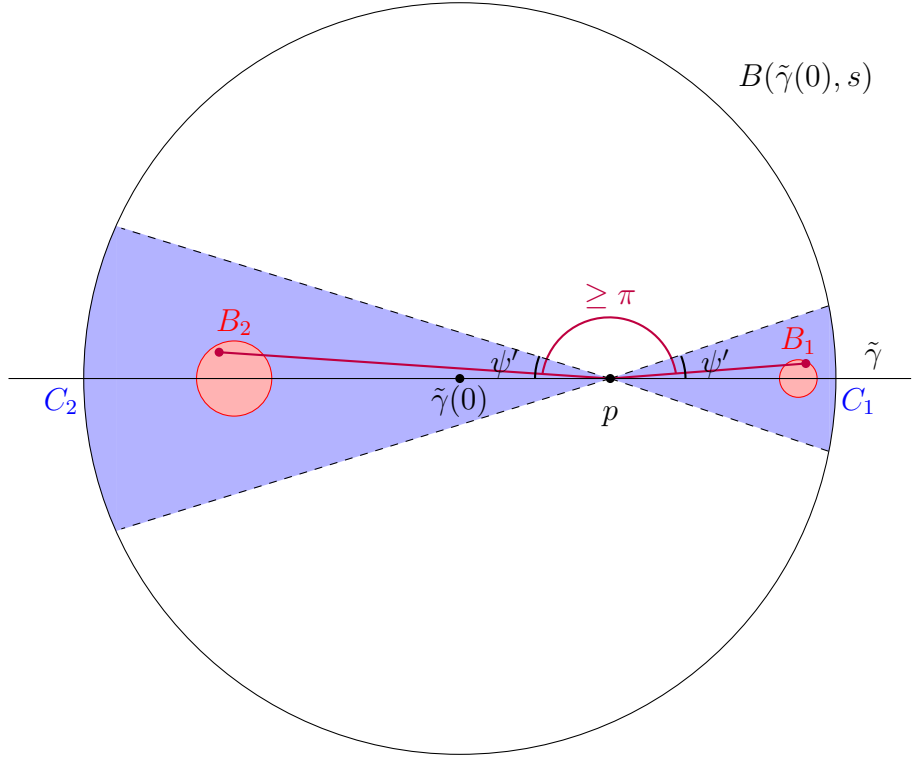
\begin{figure}[h]
\centering
\begin{tikzpicture}[scale=1] 

\fill[blue!30!white] (4.91,.97) arc (11:-11:5cm);
\fill[blue!30!white] (2,0) -- (4.91,.97) -- (4.91,-.97) -- (2,0) ; 

\fill[blue!30!white] (2,0) -- (-4.57,-2.03) -- (-4.57,2.03) -- (2,0) ; 
\fill[blue!30!white] (-4.57,-2.03) arc (204:156:5cm);

\fill[red!30!white] (4.5,0) circle(.25);
\draw[red] (4.5,0) circle(.25);

\fill[red!30!white] (-3,0) circle(.5);
\draw[red] (-3,0) circle(.5);

\draw[thick, purple] (-3.2,.35) -- (2,0) -- (4.6,.2) ;
\fill[purple] (-3.2,.35) circle(.06);
\fill[purple] (4.6,.2) circle(.06);

\draw[thick, purple] (1.1,.07) arc (170:10:.9cm);

\draw (-6,0) -- (6,0);
\fill (0,0) circle(.06);
\draw (0,0) circle(5);

\fill (2,0) circle(.06);

\draw [dashed] (2,0) -- (5,1) ;
\draw [dashed] (2,0) -- (5,-1) ;

\draw [dashed] (2,0) -- (-4.5,2) ;
\draw [dashed] (2,0) -- (-4.5,-2) ;

\draw[thick] (3,0) arc (0:22:.8cm);
\draw[thick] (1,0) arc (180:158:.8cm);

\node at (5.5,.3) {$\tilde\gamma$} ;
\node[blue] at (5.3,-.3) {$C_1$} ;
\node[blue] at (-5.3,-.3) {$C_2$} ;

\node at (4.6,4) {$B(\tilde\gamma(0),s)$} ;

\node[red] at (4.5,.5) {$B_1$} ;
\node[red] at (-3,.75) {$B_2$} ;

\node at (2,-.5) {$p$} ;
\node at (0,-.3) {$\tilde \gamma(0)$} ;

\node at (3.4,.2) {$\psi'$} ;
\node at (.6,.2) {$\psi'$} ;

\node[purple] at (2,1.1) {$\geq \pi$} ;

\end{tikzpicture}
\caption{The argument for Case 1 in Lemma~\ref{lem:lower semicts}. The geodesic segments connecting points in $B_2$ and $B_1$ meet at the cone point $p$ with angle $\geq \pi$ on both sides. Any geodesic connecting points in $B_2$ and $B_1$ must run through $p$.}\label{fig:case3}
\end{figure}

By the triangle inequality,
\begin{equation*}
d_{\tilde\SSS}(\tilde\xi(t_0+t),\tilde\gamma(c+t))\leq d_{\tilde\SSS}(\tilde\xi(t_0+t),\tilde\xi(c+t))+d_{\tilde\SSS}(\tilde\xi(c+t),\tilde\gamma(c+t))=|t_0-c|+d_{\tilde\SSS}(\tilde\xi(c+t),\tilde\gamma(c+t)).
\end{equation*} 

Let $\tilde\xi_1=g_{t_0}\tilde\xi$ and $\tilde\gamma_1=g_c\tilde\gamma$. Then, by the above inequality and Lemma~\ref{lem:Lipschitz},
\begin{equation}\label{upper_bound_xi_gamma}
d_{\GtS}(\tilde\xi_1,\tilde\gamma_1)\leq 2\delta+ e^{2|c|}\delta = (2+e^{2|c|})\delta.
\end{equation}
Moreover, for all $t\in(0,s-c]$, we obtain that
\begin{equation*}
d_{\tilde\SSS}(\tilde\xi_1(t),\tilde\gamma_1(t)) = \left\{
\begin{aligned}
&2t \qquad &\text{if} \qquad &\alpha\geq\pi,\\
& 2t\sin(\alpha/2) \qquad &\text{if} \qquad &0\leq\alpha\leq\pi,
\end{aligned}
\right.
\end{equation*}
where $\alpha$ is the (unsigned) angle between the outward trajectories of $\tilde\gamma_1$ and $\tilde\xi_1$ from the cone point $p$.

If $\alpha\geq\pi$, then $d_{\GtS}(\tilde\xi_1,\tilde\gamma_1)\geq \int_0^{s-c}2te^{-2t}dt$, which is not possible by \eqref{upper_bound_xi_gamma} and the choice of $\delta$ (see \eqref{choice delta}).  

Consider $\alpha\in[0,\pi)$. Then we have that
\begin{equation}\label{angle bound after conical}
\sin(\alpha/2)<\delta(2+e^{2|c|})\left(\int_0^{s-c}2te^{-2t}dt\right)^{-1}.
\end{equation}

Let $\beta$ be the (unsigned) angle between the inward trajectories  $\tilde\gamma_1$ and $\tilde\xi_1$ at $p$. Similarly to the argument above, we obtain that for $\delta$ as defined in \eqref{choice delta},  
\begin{equation}\label{angle bound before conical}
\sin(\beta/2)<\delta(2+e^{2|c|})\left(-\int_{-s-c}^02te^{2t}dt\right)^{-1}.
\end{equation}

Using \eqref{angle bound after conical} and \eqref{angle bound before conical}, 
\begin{equation*}
|\lambda(\gamma)-\lambda(\xi)|=\frac{1}{s}\left||\theta(\tilde\gamma,c)|-|\theta(\tilde\xi,t_0)|\right|\leq \frac{1}{s}(\alpha+\beta)\leq C\delta,
\end{equation*}
where $C=\frac{8}{s}(2+e^{2|c|})\left(\int_0^{s-|c|}2te^{-2t}dt\right)^{-1}$. Thus, for our choice of $\delta$ (see \eqref{choice delta}), we have $|\lambda(\gamma)-\lambda(\xi)|<\eps$.

\underline{Case 2}: Assume there exists $c_1\leq-s$ and $c_2\geq s$ such that $\psi_1:=|\theta(\tilde\gamma,c_1)|-\pi>0$ and $\psi_2:=|\theta(\tilde\gamma,c_2)|-\pi>0$. Denote $\tilde\gamma(c_1)=p_1$ and $\tilde\gamma(c_2)=p_2$.

Let $C_1$ be the cone around the segment $\tilde \gamma((c_1,-s])$ if $c_1\neq-s$ or $\tilde\gamma((-2s,-s))$ otherwise with vertex $p_1$ and angle $\psi'=\frac{\min\{\psi_1, \psi_2, \pi\}}{4}$. Let $C_2$ be the cone around the segment $\tilde \gamma([s,c_2])$ if $c_2\neq s$ or $\tilde\gamma((s,2s))$ otherwise, with vertex $p_2$ and angle $\psi'$. 
Set 
\begin{equation}\label{choice delta 2}
c=\min\{|c_1|,c_2\} \ \text{ and } \  \delta=\frac{1}{2}\min\left\{\frac{1}{8}e^{-2s}(c-s)\sin\psi',\quad  \min\{1,\eps c/8\}(2e^{2c}+1)^{-1}\int_{c}^{\infty}2te^{-2t}dt\right\}.
\end{equation}
Similar to Case 1, by Lemmas~\ref{lem:closeness in S and GS} and \ref{lem:Lipschitz} and the choice of $\delta$ in \eqref{choice delta 2}, if $d_{\GtS}(\tilde\gamma,\tilde\xi)<\delta$ then $\tilde\xi$ passes through $p_1$ and $p_2$. In particular, $\tilde \gamma$ and $\tilde\xi$ share a geodesic connecting $p_1$ and $p_2$. Therefore, there exists $d$ such that $g_d\tilde\xi(t)=\tilde\gamma(t)$ for $t\in[c_1,c_2]$. Let $t_1$ and $t_2$ be such that $\tilde\xi(t_1)=p_1$ and $\tilde\xi(t_2)=p_2$. Then, $|t_1-c_1|\leq 2e^{2|c_1|}\delta$ and $|t_2-c_2|\leq 2e^{2c_2}\delta$ so $|d|\leq 2e^{2c}\delta$. Moreover, by the triangle inequality, 
\begin{equation}\label{eq: shift distance}
    d_{\GtS}(g_d\tilde\xi,\tilde\gamma)\leq (2e^{2c}+1)\delta.
\end{equation}

Let $\alpha_1$ and $\alpha_2$ be the (unsigned) angles between the inward and outward trajectories of $g_d\tilde\xi$ and $\tilde\gamma$ at $p_1$ and $p_2$, respectively. Similarly to Case 1, for our choice of $\delta$, we have $0\leq \alpha_1,\alpha_2\leq\pi$,
\begin{equation*}
    \sin(\alpha_1/2)\leq\delta(2e^{2c}+1)\left(-\int_{c_2}^\infty 2te^{2t}dt\right)^{-1},
\end{equation*}
and
\begin{equation*}
    \sin(\alpha_2/2)\leq\delta(2e^{2c}+1)\left(\int_{c_1}^\infty 2te^{-2t}dt\right)^{-1}.
\end{equation*}

Therefore, 
\begin{equation*}
    |\lambda^{ss}(\gamma)-\lambda^{ss}(\xi)|\leq C\delta 
    \quad\text{and}\quad |\lambda^{uu}(\gamma)-\lambda^{uu}(\xi)|\leq C\delta,
\end{equation*}
where $C=\frac{8}{c}(2e^{2c}+1)\left(\int_{c}^\infty 2te^{-2t}dt\right)^{-1}$.

Thus, if $t_1=c_1+d\leq -s$ and $t_2=c_2+d\geq s$, then $\lambda(\xi)= \min\{\lambda^{ss}(\xi),\lambda^{uu}(\xi)\}$ and we have $|\lambda(\gamma)-\lambda(\xi)|\leq C\delta<\eps$.

Otherwise, $\lambda(\xi)\geq \min\{\lambda^{ss}(\xi),\lambda^{uu}(\xi)\}$  and we have $\lambda(\xi)\geq\lambda(\gamma)-C\delta > \lambda(\gamma)-\eps$.
\end{proof}

\begin{remark*}
Note that for this construction of $\lambda$, we do not in general have upper semicontinuity. To see this, consider a geodesic $\gamma$ which turns with angle greater than $\pi$ at times $-s$ and $c$ for some $c > 0$. Then, for all $r\in (0,s]$, $\lambda(g_{-r}\gamma) = \lambda^{ss}(g_{-r}\gamma)$, while $\lambda(\gamma) = \min\{\lambda^{ss}(\gamma),\lambda^{uu}(\gamma)\}$. Therefore, if $\lambda^{uu}(\gamma) < \lambda^{ss}(\gamma)$, we have that
$$\lambda(\gamma) = \lambda^{uu}(\gamma) < \lambda^{ss}(\gamma) = \lim_{r\downarrow 0} \lambda^{ss}(g_{-r}\gamma) = \lim_{r\downarrow 0}\lambda(g_{-r}\gamma).$$
This contradicts upper semicontinuity of $\lambda$.
\end{remark*}

Following \S3 of \cite{BCFT}, or Definition 3.4 in \cite{CT19} (and formalizing the idea presented in Section \ref{sec:defns}), we define 
\[ \GGG(\proportion) = \left\{(\gamma,t) \bigm| \int_0^\rho\lambda(g_u(\gamma))du\geq\proportion\rho \quad\text{and} \quad \int_0^\rho\lambda(g_{-u}g_t(\gamma))du\geq\proportion\rho \quad\text{for}\quad \rho\in[0,t]\right\} \]
and 
\[ \mathcal{B}(\proportion) = \left\{(\gamma,t) \bigm| \int_0^\rho\lambda(g_u(\gamma))du < \proportion\rho\right\}. \]
The decomposition we will take is $(\mathcal{P,G,S}) = (\mathcal{B}(\eta), \GGG(\eta), \mathcal{B}(\eta))$ for a sufficiently small value of $\eta$ which will be determined below. We reiterate that because of our choice of decomposition, we do not need to consider the sets of orbit segments denoted by $[\PPP],[\SSS]$, because of \cite[Lemma 3.5]{CT19}.

 While near cone points, positivity of $\lambda$ only gives us information about the closest cone point, and far from cone points, it gives us information about cone points on both sides. The following propositions help us quantify these relationships. Let $\theta_0$ be as in Lemma~\ref{lem:no big flat}(d).

%


\begin{proposition}\label{prop: distance to Con}
If $\lambda(\gamma) > \eta$, then there is a cone point in $\gamma[-\frac{\theta_0}{2\eta},\frac{\theta_0}{2\eta}]$ with turning angle at least $s\eta$ away from $\pm\pi$. In particular, if $(\gamma,t)\in \GGG(\eta)$, then there exist $t_1,t_2\in [-\frac{\theta_0}{2\eta},\frac{\theta_0}{2\eta}]$ such that $\gamma(t_1),\gamma(t+t_2)\in \operatorname{Con}$, with the turning angles at these cone points at least $s\eta$ away from $\pm\pi$.
\end{proposition}

\begin{proof}
Since $\lambda(\gamma) > \eta$, either $\lambda^{uu}(\gamma)>\eta$ or $\lambda^{ss}(\gamma)>\eta$. If $\lambda^{uu}(\gamma)> \eta$, then by Definition~\ref{defn:lambda} there is a $c\geq 0$ such that $\gamma(c)\in Con$ and  $\lambda^{uu}(\gamma) = \frac{|\theta(\gamma,c)| - \pi}{\max\{ s,c\}}$. The turning angle at $\gamma(c)$ satisfies $|\theta(\gamma,c)|-\pi \leq \theta_0/2$. Thus, 
\[\eta < \lambda^{uu}(\gamma) \leq \frac{|\theta(\gamma,c)| - \pi}{c} \leq \frac{\theta_0/2}{c}\]
and $0\leq c \leq \frac{\theta_0}{2\eta}$. Furthermore,
\[\eta < \lambda^{uu}(\gamma) \leq \frac{|\theta(\gamma,c)| - \pi}{s}\]
so the turning angle of $\gamma$ at $c$ differs from $\pi$ by at least $s\eta.$

A similar argument applies if $\lambda^{ss}(\gamma)>\eta$.
\end{proof}

Finally, we collect a statement we will need in Section~\ref{sec:pressure gap}.

\begin{lemma}\label{lem:lambda small}
Given any $\eta>0$, there exists a $\delta>0$ such that $\lambda(\gamma)<\eta$ for all $\gamma\in B(\Sing,2\delta)$.
\end{lemma}

\begin{proof}
Let $\eta>0$ be given and suppose without generality it is small enough that $\frac{s\eta}{32}<1$. We argue in $\tilde \SSS$. Suppose $\tilde \gamma\in B(\Sing,2\delta)$ and, in particular, that $\tilde \xi\in \Sing$ with $d_{\GS}(\tilde \gamma,\tilde \xi)<2\delta.$ We choose $\delta<\frac{s\theta_0}{64e^{4\theta_0/\eta}}$ and towards a contradiction suppose that $\lambda(\tilde\gamma)> \frac{\eta}{2}.$ (Recall that $s$ is specified in Lemma~\ref{lem:lower semicts}, and $\eta_0$ is specified in Lemma~\ref{lem:no big flat}(d).)

Since $\lambda(\tilde\gamma)>\frac{\eta}{2}$, by Proposition~\ref{prop: distance to Con}, there exists a cone point in $\tilde\gamma[-\frac{\theta_0}{\eta},\frac{\theta_0}{\eta}]$ at which $\tilde \gamma$ turns with angle at least $\frac{s\eta}{2}$ away from $\pm\pi$. Say $\tilde \gamma$ hits that cone point at time $t_0\in [-\frac{\theta_0}{\eta},\frac{\theta_0}{\eta}].$

As $d_{G\tilde \SSS}(\tilde \gamma, \tilde \xi)<2\delta$, by Lemma~\ref{lem:Lipschitz},
\[ d_{G\tilde \SSS}\left(g_{-\frac{2\theta_0}{\eta}}\tilde \gamma, g_{-\frac{2\theta_0}{\eta}}\tilde \xi\right) < 2\delta e^{\frac{4\theta_0}{\eta}} \ \ \mbox{ and } \ \  d_{G\tilde \SSS}\left(g_{\frac{2\theta_0}{\eta}}\tilde \gamma, g_{\frac{2\theta_0}{\eta}}\tilde \xi\right) < 2\delta e^{\frac{4\theta_0}{\eta}}.\]
Then, by Lemma~\ref{lem:closeness in S and GS},
\[ d_{\tilde\SSS}\left(\tilde \gamma\left(-\frac{2\theta_0}{\eta}\right), \tilde \xi\left(-\frac{2\theta_0}{\eta}\right)\right) < 4\delta e^{\frac{4\theta_0}{\eta}}  \ \ \mbox{ and } \ \   d_{\tilde\SSS}\left(\tilde \gamma\left(\frac{2\theta_0}{\eta}\right), \tilde \xi\left(\frac{2\theta_0}{\eta}\right)\right) < 4\delta e^{\frac{4\theta_0}{\eta}}. \]

Consider the geodesic segment $c$ connecting $\tilde \xi(-\frac{2\theta_0}{\eta})$ and $\tilde \gamma(t_0)$. The segment $c$ and $\tilde\gamma[-\frac{2\theta_0}{\eta}, t_0]$ agree at $t_0$ and at time $-\frac{2\theta_0}{\eta}$, at least $\frac{\theta_0}{\eta}$ away with respect to $d_{\tilde \SSS}$, are at most $4\delta e^{\frac{4\theta_0}{\eta}}$ apart. Comparing to a Euclidean triangle and using the CAT(0) property, the angle between these segments at $\tilde \gamma(t_0)$ is at most $2\sin^{-1}[(4\delta e^{\frac{4\theta_0}{\eta}})/(2\frac{\theta_0}{\eta})]$. By our choice of $\delta$, this is less than $2\sin^{-1}[\frac{s\eta}{32}]<\frac{s\eta}{8}$. The same argument applies to the angle between $\tilde \gamma[t_0, \frac{2\theta_0}{\eta}]$ and the segment $c'$ from $\tilde \gamma(t_0)$ to $\tilde\xi(\frac{2\theta_0}{\eta})$.

At $t_0$, $\tilde \gamma$ turns with angle at least $\frac{s\eta}{2}$ away from $\pm\pi$. Therefore, the concatenation of $c$ with $c'$ turns with angle at least $\pi +\frac{s\eta}{4}$ on both sides, and hence is geodesic. By uniqueness of geodesic segments in $\tilde \SSS$, $\xi[-\frac{2\theta_0}{\eta},\frac{2\theta_0}{\eta}]$ must agree with this concatenation. But this contradicts the fact that $\xi \in \Sing$.  Therefore, $\lambda(\gamma)\leq\frac{\eta}{2}<\eta$ as desired.
\end{proof}

%

\section{$\GGG(\proportion)$ has weak specification (at all scales)}\label{sec:specification}

The goal of this section is to obtain Corollary~\ref{weak specification} which shows that $\GGG(\proportion)$ has weak specification at all scales.




\begin{lemma}(Compare with Lemma 3.8 in \cite{Dankwart})\label{density of saddle connections}
Let $x\in\SSS$ and $\beta$ be a geodesic ray with $\beta(0)=x$. Then, for any $\eps>0$ there exist $T_0(\eps)$ and a geodesic $c$ which connects $x$ with a point $z\in Con$ so that the length of $c$ is at most $T_0(\eps)$ and $\angle_x(\beta,c)<\eps$ where $\angle_x(a,b)$ is the angle at $x$ between geodesic segments $a$ and $b$.
\end{lemma}

\begin{proof}
Let $\tilde x\in\tilde\SSS$ be a lift of $x$ and $\tilde \beta $ a lift of $\beta$ with $\tilde \beta(0)=\tilde x$. Denote by $C_{\frac{\eps}{2}}(\tilde x, \tilde \beta)$ the cone around $\beta$ with vertex $\tilde x$ and angle $\frac{\eps}{2}$. Choose $T_0=T_0(\eps)$ so large that an angle-$\eps$ sector of a radius-$T_0$ Euclidean ball contains a ball of radius much larger than the diameter of $\SSS$. Then $I=B_{T_0}(\tilde x)\cap C_{\frac{\eps}{2}}(\tilde x, \tilde \beta)$ is at least as large as this Euclidean sector and so must contain a fundamental domain of $\SSS$. Then $\widetilde{Con}\cap Int(I)\neq \emptyset$, so let $\tilde z\in \widetilde{Con}\cap Int(I)$ such that $\tilde z$ is closest to $\tilde x$. The segment $\tilde c=\tilde x \tilde z$ is a geodesic of length at most $T_0$. The projection of $\tilde c$ to $\SSS$ is the desired geodesic.
\end{proof}

\begin{lemma}\label{make endpoints in Con}
For any $\delta>0$, there exists $T_1=T_1(\delta)$ such that for any $t>0$ and $(\gamma,t)\in\GGG(\proportion)$, $(\gamma,t)$ is $\delta$-shadowed by a saddle connection path $\gamma_e$ in the following sense: 
\begin{itemize}
    \item $\ell(\gamma_e)\leq t+2T_1$
    \item there exists $s_0\in[0,T_1]$ with the property that if $\gamma^c_e$ is any extension of $\gamma_e$ to a complete geodesic then $d_{\GS}(g_u(\gamma), g_u(g_{s_0}(\gamma^c_e)))\leq\delta$ for all $u\in[0,t]$.
\end{itemize}
In particular, if $t > \frac{\theta_0}{\eta}$, there exists a closed interval $I\supset [\frac{\theta_0}{2\eta},t-\frac{\theta_0}{2\eta}]$ such that $\gamma_e(s_0+u)=\gamma(u)$ for $u\in I$.
\end{lemma}

\begin{proof}
As usual, we prove the result in $\tilde \SSS$. Let $T=\max\{-\log(\delta),\frac{\theta_0}{2\eta}, \frac{\ell_0}{4}\}$ where $\theta_0$ and $\ell_0$ are from Lemma~\ref{lem:no big flat}. By Lemma~\ref{lem:shadow in S}, if we construct $\tilde\gamma_e$ such that $d_{\tilde\SSS}(\tilde\gamma(u),\tilde\gamma_e(s_0+u))<\frac{\delta}{2}$ for all $u\in[-T, t+T]$, then $d_{\GS}(g_u(\gamma), g_u(g_{s_0}(\gamma^c_e)))\leq\delta$ for all $u\in[0,t]$. (See Figure~\ref{fig:gamma e} for the constructions in this proof.)

\begin{figure}[h]
\centering
\begin{tikzpicture}[scale=1]

\fill[blue!20!white] (1,0) -- (-6,1.5) -- (-6,-2)  ;
\fill[blue!40!white] (-1,-.3) -- (-6,-.5) -- (-6,-1.5)  ;
\fill[blue!60!white] (-3,-.5) -- (-6,-.7) -- (-6,-1.3)  ;

\draw[thick] (-6,0) -- (5,0);
\draw[thick, red] (-5.5,-1) -- (-3,-.5) -- (-1,-.3) -- (1,-.02) -- (5,-.02);

\node at (6,0) {$\cdots$} ;

\draw[thick, purple] (1.3,0) arc (0:158:.3cm);
\draw[thick, purple] (1.3,0) arc (0:-158:.3cm);

\fill (1,0) circle(.06);
\draw[thick, ->] (3,0) -- (3.5,0) ;
\draw[thick, red, ->] (3,-.05) -- (3.5,-.05) ;

\fill (3,0) circle(.06);
\fill (-4,0) circle(.06);
\fill (-1,-.3) circle(.06);
\fill (-3,-.5) circle(.06);
\fill (-5.5,-1) circle(.06);

\draw [dashed] (1,0) -- (-3.5,2) ;
\draw [dashed] (1,0) -- (-3.5,-2) ;

\node at (3.2,.5) {$\tilde \gamma$} ;
\node[red] at (3.2,-.5) {$\tilde \gamma_e$} ;
\node at (1,1) {$\tilde \gamma(s_1)$} ;
\node at (-4,.5) {$\tilde \gamma(-T)$} ;
\node at (-1,-.7) {$\tilde p_1$} ;
\node at (-3,-1.1) {$\tilde p_2$} ;

\node[purple] at (1.3,.4) {$\pi$} ;
\node[purple] at (1.3,-.4) {$\pi$} ;

\node[blue] at (-6,1.5) {$\mathcal{C}$} ;

\end{tikzpicture}
\caption{The construction of $\gamma_e$ in Lemma~\ref{make endpoints in Con} around the left endpoint of $\gamma$. The sequence of $\alpha$-cones featured in the proof is shaded.}\label{fig:gamma e}
\end{figure}
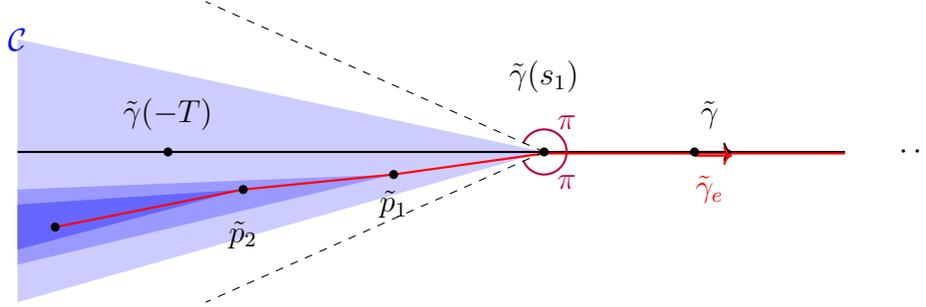

By Proposition~\ref{prop: distance to Con}, there exist $t_0,t_1\in[-\frac{\theta_0}{2\eta},\frac{\theta_0}{2\eta}]$ such that $\tilde\gamma(t_0), \tilde\gamma(t+t_1)\in \widetilde{Con}$, $|\theta(\tilde\gamma,t_0)|-\pi\geq s\eta$ and $|\theta(\tilde\gamma,t+t_1)|-\pi\geq s\eta$. Thus, there exist $s_1\in [-T,\frac{\theta_0}{2\eta}]$ and $s_2\in[-\frac{\theta_0}{2\eta},T]$ such that $\tilde\gamma(s_1), \tilde\gamma(t+s_2)\in \widetilde{Con}$ and $\left(\tilde\gamma([-T,s_1))\cup \tilde\gamma((t+s_2,t+T])\right)\cap \widetilde{Con}=\emptyset$.

If $s_1=-T$, then define $\tilde\gamma_e(u-s_1)=\tilde\gamma(u)$ for $u\in[s_1,t+s_2]$. Assume $s_1>-T$. Let $\eta_0$ be as in Lemma~\ref{lem:no big flat}(b). Choose $\alpha<\frac{\ell_0}{4(T+\frac{\theta_0}{2\eta})}\min\{\eta_0,\frac{\delta}{2(\frac{\theta_0}{2\eta}+T)}\}$.  Let $\mathcal C$ be the cone in $\tilde \SSS$ around $\tilde\gamma([-T,s_1])$ with angle $\alpha/2$. Note that $\alpha<\eta_0$, so any geodesic segment from a point in $\mathcal{C}$ to $\tilde\gamma(s_1)$ can be concatenated with $\tilde \gamma([s_1,t])$ to form a geodesic. By Lemma~\ref{density of saddle connections}, there exists $T_0=T_0(\frac{\alpha}{2})\geq T+\frac{\theta_0}{2\eta}$ and a point $\tilde p_1$ in $\widetilde{Con} \cap \mathcal{C}$ such that $d_{\tilde\SSS}(\tilde p_1,\tilde\gamma(s_1))\leq T_0$. Choose $\tilde p_1$ as in the previous sentence minimizing the distance to $\tilde\gamma([-T_0,s_1])$. If $d_{\tilde\SSS}(\tilde p_1,\tilde\gamma(s_1))\geq T+\frac{\theta_0}{2\eta}$, then let the initial segment of $\tilde \gamma_e$ be the geodesic segment $[\tilde p_1,\tilde \gamma(s_1)]$.

Otherwise, we repeat the argument above, applying Lemma~\ref{density of saddle connections} to construct an angle-$\alpha/2$ cone centered around the geodesic segment making angle $\pi+\frac{\alpha}{2}$ with $[\tilde p_1,\tilde \gamma(s_1)]$. We get a point $\tilde p_2\in \widetilde{Con}$ in this cone with $d_{\tilde S}(\tilde p_1,\tilde p_2)\leq T_0$, again chosen to minimize the distance to $\tilde \gamma([-T_0,s_1])$. If $d_{\tilde S}(\tilde p_2, \tilde \gamma(s_1))\geq T+\frac{\theta_0}{2\eta}$, then let the initial segment of $\tilde\gamma_e$ be the concatenation of geodesic segments $[\tilde p_2, \tilde p_1]$ and $[\tilde p_1, \tilde \gamma(s_1)]$. This concatenation is a geodesic by the choice of $\alpha$ and the construction of the cone. Otherwise, repeat the procedure at $\tilde p_2$ and so on. 

We will need to repeat this procedure at most $\frac{T+\frac{\theta_0}{2\eta}}{\ell_0}$ times. We extend the beginning of $\tilde \gamma_e$ constructed here with $[\tilde \gamma(s_1),\tilde \gamma(t+s_2)]$ and then extend beyond $\tilde \gamma(t+s_2)$ (if needed) similarly to the procedure at $\tilde \gamma(s_1)$. Since the turning angles at each cone point are at least $\pi$, we obtain a saddle connection path $\tilde \gamma_e$.

Let $T_1=T+\frac{\theta_0}{2\eta}+T_0$. Let $s_0$ be such that $\tilde\gamma_e(s_0+s_1) = \tilde\gamma(s_1)$. Then $s_0\in[0,T_1]$. For $u\in [s_1, s_2]$, $d_{\tilde S}(\tilde\gamma(u),\tilde\gamma_e(s_0+u))=0<\frac{\delta}{2}$, as desired. For $u\in[-T,s_1]$, note that the sequence of cones used in the proof have angles $\frac{\alpha}{2}$ and $\alpha<\frac{\ell_0}{4(T+\frac{\theta_0}{2\eta})}\frac{\delta}{2(\frac{\theta_0}{2\eta}+T)}$. There are at most $\frac{T+\frac{\theta_0}{2\eta}}{\ell_0}$ of these cones, each segment from $\tilde p_i$ to $\tilde p_{i+1}$ is at most length $T+\frac{\theta_0}{2\eta}$, and we always choose our cone points $\tilde p_i$ as close to $\tilde \gamma$ as we can. Therefore, the distance $d_{\tilde S}(\tilde \gamma(u), \tilde\gamma_s(s_0+u)$ is bounded by $\frac{\delta}{2}$ for $u\in [-T,s_1]$. For the same reason, this bound also holds for $u\in [s_2, t+T]$, finishing the proof. 
\end{proof}

\begin{lemma}(Compare with Lemma 3.9 in \cite{Dankwart})\label{connections at Con}
Let $N=[\frac{4\pi}{\eta_0}]+3$ where $\eta_0$ is from Lemma~\ref{lem:no big flat}(c). Let $q\in Con$. Then there exist $N$ saddle connections $\sigma_1,\sigma_2,\ldots, \sigma_N$ emanating from $q$ with the following property:

For any geodesic segment $\gamma$ with endpoint $q$, the concatenation of $\gamma$ with at least one $\sigma_i$ is also a local geodesic.
\end{lemma}

\begin{proof}
We have $\mathcal L(q)=2\pi+\alpha \geq 2\pi+\eta_0$. Divide the space of directions at $q$ into intervals of size no more than $\frac{\alpha}{2}$; at most $\lceil \frac{2\pi+\alpha}{\alpha/2}\rceil \leq N$ intervals are needed. Using Lemma~\ref{density of saddle connections}, pick a saddle connection emanating from $q$ with direction in each of these intervals. These are the $\sigma_i$.

The concatenation of $\gamma$ and some saddle connection $\sigma_i$ is a geodesic if and only if $c_i$ lies outside of the $\pi$-cone of directions at $q$ around $\gamma$. The complement of this cone in the space of directions at $q$ is an interval of size $\mathcal L(q)-2\pi=\alpha$ and must therefore fully contain one of our $\frac{\alpha}{2}$-size intervals. The $\sigma_i$ chosen in this interval geodesically continues $\gamma$ as desired. 
\end{proof}

\begin{lemma}(Compare with Corollary 3.1 in \cite{Dankwart})\label{connecting saddle connections with arbitrary length}
For any two parametrized saddle connections $\sigma,\sigma'$ on $\SSS$ there exists a geodesic segment $\gamma$ which first passes through $\sigma$ and eventually passes through $\sigma'$.    
\end{lemma}

\begin{proof}
Let $\alpha$ be a closed geodesic which turns with angle greater than $\pi$ at a cone point $p$ (such $\alpha$ exists by Lemma~\ref{existence of a turning closed geodesic}). Denote by $\tilde\sigma$ the lift of $\sigma$ to $\tilde\SSS$ which has the starting point $\tilde a$ and the endpoint $\tilde b$. Consider a parametrized complete lift $\tilde\alpha$ of $\alpha$ such that it is disjoint from $\tilde\sigma$ and its positive endpoint is contained in the complement of the cone around $[\tilde a, \tilde b]$ with vertex $\tilde b$ and angle $\pi$. 
Denote by $\tilde c_t\colon [0,\ell_t]\rightarrow\tilde\SSS$ the geodesic that connects $\tilde a$ with $\tilde\alpha(t)$. By the choice of the lift $\tilde\alpha$, there is a time $t_0$ such that for all $t\geq t_0$, $\tilde c_t$ passes through $\tilde b$ and that $\tilde c_{t_0}$ only shares its endpoint with $\tilde\alpha$. 

We now need the following fact.

\begin{sublemma}\label{positive intersection}
There exists $t_1>t_0$ such that $\tilde c_{t_1}$ intersects the geodesic segment $[\tilde \alpha(t_0), \tilde \alpha(t)]$ in a positive-length segment.  
\end{sublemma}

\begin{proof}[P\textbf{roof of Sublemma.}]
Consider the geodesic triangle in $\tilde \SSS$ with vertices $\tilde a$, $\tilde \alpha(t_0)$ and $\tilde \alpha(t)$ for $t>t_0$. As $t$ increases, the length of the side $[\tilde \alpha(t_0), \tilde \alpha(t)]$ increases without bound while the length of $[\tilde a, \tilde c(t_0)]$ is fixed, so the length of $\tilde c_{t}=[\tilde a, \tilde \alpha(t)]$ must also increase without bound once $t$ is sufficiently large, by the triangle inequality. The comparison triangles in $\mathbb{R}^2$ will have one side of fixed length while the other two become very long. The angle at the vertex of the comparison triangle where the long sides meet must therefore become arbitrarily small.

At each lift of the cone point $p$ that $\alpha$ passes through, $\tilde \alpha$ has turning angle $\pi+\theta$ for some $\theta>0$. Let $T$ be so large that the angle noted above in the Euclidean comparison triangle is $<\theta$.  As $\tilde \SSS$ is CAT(0), the original triangle in $\tilde \SSS$ has angles no larger than those in the comparison triangle. Thus, the angle between $\tilde c_t$ and $[\tilde \alpha(t_0), \tilde \alpha(t)]$ will be less than $\theta$ for all $t\geq T$. Let $t'$ be any time greater than $T$ at which $\tilde \alpha$ passes through a lift of the cone point and let $t_1>t'$. Since $\tilde \alpha$ turns with excess angle $\theta$ at $\tilde \alpha(t')$, the concatenation of $\tilde c(t')$ and $\tilde \alpha([t',\infty))$ is a geodesic ray. Therefore $\tilde c_{t_1}$ and $[\tilde \alpha(t_0), \tilde \alpha(t)]$ intersect in a positive-length segment.
\end{proof}

For $t_1$ as in the Sublemma, the projection of $\tilde c_{t_1}$ to $\SSS$ is a local geodesic which first passes through $\sigma$ and eventually through a piece of $\alpha$. By extending the resulting local geodesic along $\alpha$, we can make sure that it passes through the whole curve $\alpha$. We denote the resulting local geodesic by $g_1$. 

We apply the above argument to $\sigma'$ and $\alpha$ with their orientations reversed to obtain a local geodesic $g_2$ that connects these curves.

The concatenation of $g_1$ and $g_2$ (with its orientation reversed) has the desired property.
\end{proof}

Repeating the proof of Proposition 3.2 in \cite{Dankwart}, replacing \cite[Lemma 3.9]{Dankwart} by Lemma~\ref{connections at Con} and \cite[Corollary 3.1]{Dankwart} by Lemma~\ref{connecting saddle connections with arbitrary length}, we obtain Proposition~\ref{connecting saddle connections} which strengthens Lemma~\ref{connecting saddle connections with arbitrary length}. We include the proof of the proposition for completeness.

\begin{proposition}(Compare with Proposition 3.2 in \cite{Dankwart})\label{connecting saddle connections}
There exists a constant $C(\SSS)>0$ so that the following holds:

For any two parametrized saddle connections $\sigma,\sigma'$ on $\SSS$ there exists a geodesic segment $\gamma$ which first passes through $\sigma$ and eventually passes through $\sigma'$ and which is of length at most $C(\SSS)+\ell(\sigma)+\ell(\sigma')$.
\end{proposition}

\begin{proof}
Recall that $\SSS$ has only finitely many cone points. By Lemma~\ref{connections at Con}, there are $N_0=N_0(\SSS)$ parametrized saddle connections $\sigma_1,\sigma_2,\ldots,\sigma_{N_0}$ with the property that for any geodesic segment with endpoint in $Con$ (in particular, any saddle connection) the concatenation of it with at least one $\sigma_i$ is a local geodesic. By Lemma~\ref{connecting saddle connections with arbitrary length}, for each pair $(\sigma_i,\sigma_j)$ there is a local geodesic $c_{ij}$ which first passes through $\sigma_i$ and eventually through $\sigma_j$. Since there are only finitely many pairs $(\sigma_i, \sigma_j)$, there exists a constant $C(\SSS)$ such that $\ell(c_{ij})\leq C(\SSS)$. Thus, for any two parametrized saddle connections $\sigma,\sigma'$ we do the following. First, we connect the endpoint of $\sigma$ to $\sigma_i$ for some $i$ and the starting point of $\sigma'$ (the endpoint of the saddle connection with the reversed parametrization of $\sigma'$) to $\sigma_j$ for some $j$ so that the results of concatenations are local geodesics. Then, the concatenation of $\sigma$ with $c_{ij}$ followed by the concatenation with $\sigma'$ is the desired geodesic segment which first passes through $\sigma$ and eventually through $\sigma'$ of length at most $C(\SSS)+\ell(\sigma)+\ell(\sigma')$. 
\end{proof}

Using Lemma~\ref{make endpoints in Con} and Proposition~\ref{connecting saddle connections}, we obtain the weak specification property on $\GGG(\proportion)$ at all scales.

\begin{corollary} (Weak specification)\label{weak specification}
For all $\delta>0$ there exists $T=T(\eta,\delta,\SSS)>0$ such that for all $(\gamma_1,t_1), \ldots, (\gamma_k,t_k)\in\GGG(\proportion)$ there exist $0=s_1<s_2<\ldots<s_k$ and a geodesic $\gamma$ on $\SSS$ such that for all $i=1,\ldots, k$ we have $s_{i+1}-(s_i+t_i)\in[0,T]$ and $d_{\GS}(g_u(\gamma_i), g_u(g_{s_i}(\gamma))<\delta$ for all $u\in[0,t_i]$.
\end{corollary}

\begin{proof}
We can take 
$T=2T_1+C(\SSS)$ where 
$T_1$ is as in Lemma~\ref{make endpoints in Con} and $C(\SSS)$ is as in Proposition~\ref{connecting saddle connections}. We omit the proof here as it is a simplified version of the proof of Proposition~\ref{specification}. 
\end{proof}

%

\section{$\GGG(\eta)$ has strong specification (at all scales)}\label{sec:strong}

The goal of this section is to upgrade the weak specification property of Corollary~\ref{weak specification} to strong specification (Proposition~\ref{specification}), in which we have more precise control over when our shadowing geodesic shadows each segment. 

As $\eta$ is fixed throughout, we write $\GGG := \GGG(\eta)$.

\begin{lemma}\label{lem: eventual density}
If $G\subset \mathbb{R}^{\geq 0} \not\subset c\mathbb{N}$ for all $c > 0$, then for all $\delta > 0$, there exist $x, y\in G$ and $n,m\in\mathbb{N}$ such that $0 < nx - my < \delta$.
\end{lemma}

\begin{proof}
Let $x$ denote the smallest non-zero element of $G$, which exists, as otherwise we are immediately done. Now, there are three cases. 
	
First, assume there exists $y\in G$ such that $\frac{y}{x}\notin\mathbb{Q}$. Now take $q\in\mathbb{N}$ large enough so that $\frac{x}{q} < \delta$, and so that there is $p\in\mathbb{N}$ with $\lvert \frac{y}{x} - \frac{p}{q}\rvert < \frac{1}{q^2}$ by Dirichlet's theorem. Then, this implies that
	$$\left\lvert qy - px\right\rvert < \frac{x}{q} < \delta.$$
In the second case, suppose that for all $y\in G$, $\frac{y}{x}$ is rational, and when written in lowest terms, the denominators can be arbitrarily large. Then, take $n$ such that $\frac{x}{n} < \delta$ and $y\in G$ with $\frac{y}{x} = \frac{p}{q}$ in lowest terms for some $q > n$. Then, as $p$ is invertible in $\mathbb{Z}/q\mathbb{Z}$, we can take $m$ to be a positive integer such that $mp = 1\pmod q$. It follows that
	$$\left\lvert \frac{mp-1}{q}x - my \right\rvert = \frac{x}{q} < \delta.$$
	
Finally, in the third case, $\frac{y}{x}$ is always rational, but with denominators bounded above by $M$. Then, $G\subset \frac{x}{M!}\mathbb{N}$, a contradiction.
\end{proof}

\begin{lemma}\label{lem: delta-dense}
Suppose $x > y > 0$ and $x - y = \delta$. Then, there exists $T > 0$ such that for all $\tau\geq T$ and all $n\in\mathbb{N}\cup\{0\}$, there exists $m_1,m_2 \in\mathbb{N}$ such that $\tau + n\delta \leq m_1x + m_2y \leq \tau + (n+1)\delta$.
\end{lemma}

\begin{proof}
Fix $C$ such that $C > \frac{y}{\delta} + 2$. We claim that $T = \max\{Cy,1\}$. Fix $\tau\geq T$. Now, let $n\in\mathbb{N}\cup\{0\}$. Fix $k_1$ to be the largest integer such that $k_1y \leq \tau + n\delta$ and then choose $k_2$ to be the smallest positive integer such that $k_1y + k_2\delta \geq \tau + n\delta$. Therefore, we see that $k_2x + (k_1 - k_2)y = k_1y + k_2\delta$, and so
		\[\tau+n\delta \leq k_2x + (k_1 - k_2)y \leq \tau + (n+1)\delta.\]
Observe that by construction,
		\[k_1y + (k_2 - 1)\delta < \tau + n\delta < k_1y + y,\]
and consequently, $k_2 < \frac{y}{\delta} + 1$. Therefore, by our choices of $\tau$ and $C$,
		\[k_1 > \frac{\tau + n\delta - y}{y} > \frac{Cy - y}{y} >  \frac{y}{\delta} + 1.\]
Thus, $k_1 - k_2 > 0$, and we are done.
\end{proof}

We need the following result of Ricks; we explain the necessary terminology in the course of applying it:

\begin{theorem}\cite[Theorems 4 and 5]{R17}\label{Ricks result}
Let $X$ be a proper, geodesically complete, CAT(0) space under a proper, cocompact, isometric action by a group $\Gamma$ with a rank one element, and suppose $X$ is not isometric to the real line. Then, the length spectrum is arithmetic if and only if there is some $c > 0$ such that $X$ is isometric to a tree with all edge lengths in $c\mathbb{Z}$.
\end{theorem}

\begin{proposition}\label{geodesics of similar lengths}
Given $\delta > 0$, there exist two closed saddle connection paths $\gamma,\xi$ such that $0<\lvert \ell(\gamma) - \ell(\xi)\rvert < \delta.$
\end{proposition}

\begin{proof}
This follows for translation surfaces by combining Lemma~\ref{lem: eventual density} with $\S 6$ of \cite{Colognese} (see hypothesis (T3) and the discussion following \cite[Proposition 6.9]{Colognese}).

For general flat surfaces with conical points, this follows from Theorem~\ref{Ricks result}. We outline the reasoning as follows. We say that $\gamma\in \Gamma$ is rank one if there exists a geodesic $\eta$ such that $\gamma\eta = g_t\eta$ for some $t > 0$ and $\eta$ does not bound a flat half plane, i.e., a subspace isometric to $\mathbb{R}\times [0,\infty)$. The existence of this follows from the existence of a closed geodesic which turns with angle greater than $\pi$ at some cone point (see Lemma~\ref{existence of a turning closed geodesic}). Now, the universal cover of a flat surface with cone points is not isometric to a tree with edge lengths in $c\mathbb{Z}$, and so it follows that the length spectrum is not arithmetic. The length spectrum is the collection of lengths of hyperbolic isometries in $\Gamma$, which is precisely the set of lengths of closed geodesics, which by Lemma~\ref{lem:closed saddle} is the set of lengths of closed saddle connection paths. We can now apply Lemma~\ref{lem: eventual density}.
\end{proof}

\begin{proposition}\label{prop: update saddle connections}
For all $\delta > 0$, there exists $\tau=\tau(\delta) > 0$ and $\delta' < \delta$ such that for any $\tau'>\tau$, any two saddle connections $\sigma$, $\sigma'$ and any $n\in\mathbb N\cup\{0\}$, there exists a geodesic segment $\xi_n$ which begins with $\sigma$ and ends with $\sigma'$ with length in $[\ell(\sigma) + \ell(\sigma') + \tau' + n\delta', \ell(\sigma) + \ell(\sigma') + \tau' + (n+1)\delta']$.
\end{proposition}

\begin{proof}
Fix $\delta > 0$ and take $\gamma_1,\gamma_2$ to be closed geodesics such that $0<\lvert \ell(\gamma_1) - \ell(\gamma_2)\rvert = \delta' < \delta$, which exist by Proposition~\ref{geodesics of similar lengths}. Now take $\tau = 3C(\SSS) + T$, where $C(\SSS)$ is from Proposition~\ref{connecting saddle connections} and $T$ is from Lemma~\ref{lem: delta-dense} applied for $\ell(\gamma_1)$ and $\ell(\gamma_2)$.

Consider two saddle connections $\sigma$ and $\sigma'$ and apply Proposition~\ref{connecting saddle connections} three times to connect, in sequence, $\sigma$ to $\gamma_1$ to $\gamma_2$ to $\sigma'$ with the geodesic $\xi$. Furthermore, $\ell(\xi) = L + \ell(\sigma) + \ell(\gamma_1) + \ell(\gamma_2) + \ell(\sigma')$ and $L \leq 3C(\SSS)$. Because the $\gamma_i$ are closed geodesics, there is a geodesic $\xi_{k_1,k_2}$ which follows the exact path of $\xi$ except that it loops around $\gamma_i$ a total of $k_i$ times. In other words, $\ell(\xi_{k_1,k_2}) = \ell(\xi) + (k_1 - 1)\ell(\gamma_1) + (k_2 - 1)\ell(\gamma_2)$. Now let $n\in\mathbb{N}$, and, using Lemma~\ref{lem: delta-dense}, take $k_1,k_2$ such that
$$k_1\ell(\gamma_1) + k_2\ell(\gamma_2) \in [T + (3C(\SSS) - L)+(\tau'-\tau) + n\delta', T + (3C(\SSS) - L)+(\tau'-\tau) + (n+1)\delta'].$$
Then $\xi_n := \xi_{k_1,k_2}$ satisfies our desired property.
\end{proof}






\begin{proposition}\label{specification}
The collection of orbit segments $\GGG=\GGG(\proportion)$ has strong specification at all scales. That is, for any $\eps>0$, there exists $\hat\tau(\eps) > 0$ such that for any finite collection $\{(\gamma_i,t_i)\}_{i=1}^n\subset \GGG$, there exists $\hat\xi\in \GS$ that $\eps$-shadows the collection with transition time $\hat\tau$ between orbit segments. In other words, for $1\leq i \leq n$, 
	$$d_{\GS}(g_{u + \sum_{j=1}^{i-1} (t_j + \hat\tau)} \hat\xi, g_u \gamma_i) \leq \eps \text{ for } 0\leq u \leq t_i.$$
Moreover, for $1\leq i\leq n$ such that $t_i > \frac{\theta_0}{\eta}$ where $\theta_0$ as in Lemma~\ref{lem:no big flat}(d), there exists a closed interval $I_i\supset [\frac{\theta_0}{2\eta},t_i-\frac{\theta_0}{2\eta}]$ such that $\hat\xi(u+\sum_{j=1}^{i-1} (t_j + \hat\tau))=\gamma_i(u)$ for $u\in I_i$.
\end{proposition}
\begin{proof}
 By Lemma~\ref{make endpoints in Con}, there exists $T_1=T_1(\frac{\eps}{2},\SSS)$ for each $i=1,\ldots, n$, there exists a saddle connection path $\gamma^e_i$ such that $\ell(\gamma^e_i)\leq t_i+2T_1$ and there exists $s_{i}\in[0, T_1]$ such that for any extension $\hat\gamma^e_i$ of $\gamma^e_i$ to a complete geodesic we have
\begin{equation*}
    d_{\GS}(g_u(\gamma_i),g_u(g_{s_{i}}(\hat\gamma^e_i)))\leq\frac{\eps}{2} \quad\text{ for  all}\quad u\in[0,t_i].
\end{equation*}
Moreover, if $t_i > \frac{\theta_0}{\eta}$, there exists a closed interval $I_i\supset [\frac{\theta_0}{2\eta},t_i-\frac{\theta_0}{2\eta}]$ such that $\gamma_i^e(s_i+u)=\gamma_i(u)$ for $u\in I_i$.
We will construct our shadowing geodesic by induction. Let $\tau=\tau\left(\frac{\eps}{4}\right)$, $\delta'<\frac{\eps}{4}$ be as in Proposition~\ref{prop: update saddle connections} applied for $\delta=\frac{\eps}{4}$. Denote $T=\tau+3T_1$. 

Thus, for any $k=1,\ldots, n-1$ and $m_k\in\mathbb N\cup\{0\}$, there exists a geodesic segment $\xi_{k+1}$ which begins with $\gamma^e_k$ and ends with $\gamma^e_{k+1}$ with length $\ell(\xi_{k+1}) = \ell(\gamma^e_k) + \ell(\gamma^e_{k+1}) + T-(s_{k+1}-s_k)-(\ell(\gamma^e_k)-t_k)+c_k$ where $c_k\in[m_k\delta',(m_k+1)\delta']$.

Moreover, by Lemma~\ref{make endpoints in Con}, for any extension $\hat\xi_{k+1}$ of $\xi_{k+1}$ to a complete geodesic with $\hat\xi_{k+1}(u)=\xi_{k+1}(s_k+u)$ for all $u\in[-s_k,-s_k+\ell(\xi_{k+1})]$, we have
\begin{align}\label{ineq: segment}
&d_{\GS}(g_u\hat\xi_k, g_u\gamma_k)\leq \frac{\eps}{2} \quad \text{ for all}\quad u\in[0,t_k]\quad \text{and}\quad\nonumber\\ &d_{\GS}(g_u(g_{t_k+T+c_k}\hat\xi_{k+1}), g_u\gamma_{k+1})\leq \frac{\eps}{2} \quad \text{ for all}\quad u\in[0,t_{k+1}].
\end{align}

We define the sequence $m_k$ inductively. Let $m_1=0$. In particular, $c_1\in[0,\delta']\subset[0,\frac{\eps}{4}]$. For $k>1$, we set $$m_k = \lceil\frac{\eps}{4\delta'}\rceil \quad \text{if}\quad (k-1)\frac{\eps}{4} - \sum_{i \leq k-1} c_i > \frac{\eps}{4},\quad \text{and}\quad 0\quad \text{otherwise},$$ as this ensures $\left|\sum_{j=1}^{k-1} \frac{\eps}{4} - c_j\right| < \frac{\eps}{2}$.

Let $\xi$ be a geodesic segment that is a result of gluing $\xi_{k}$ and $\xi_{k+1}$ along $\gamma^e_{k}$ that is the end of $\xi_k$ and the beginning of $\xi_{k+1}$ for all $k=2,\ldots, n-1$. Let $\hat\xi$ be any extension of $\xi$ to a complete geodesic with the parametrization such that $\hat\xi(-s_1)=\xi(0)$. By the choice of $m_k$ and \eqref{ineq: segment}, we obtain for $1 \leq i \leq n$,
$$d_{\GS}(g_u(g_{\sum_{j=1}^{i-1}(t_{j} + T + \eps/4)}\hat\xi), g_u\gamma_i)\leq d_{\GS}(g_u(g_{\sum_{j=1}^{i-1}(t_{j} + T + c_j)}\hat\xi), g_u\gamma_i)+\frac{\eps}{2}\leq \eps  \quad \text{ for all}\quad u\in[0,t_i].$$

Thus, $\hat\xi$ is the desired shadowing geodesic. As a result, the collection of orbit segments $\GGG$ has strong specification at all scales with the specification constant $T+\frac{\eps}{4}$.
\end{proof}

We close this section by recording a simple technical modification of Proposition~\ref{specification} which we will need when we apply specification in Section~\ref{sec:pressure gap}.

\begin{definition}\label{defn: good bounded shift}
Let $M>0$ and $\eta>0$ be given. We denote by $\mathcal{G}^M(\eta)$ the set of all orbit segments $(\gamma,t)$ such that there exist $t_1, t_2$ with $|t_i|<M$ such that $(g_{t_1}\gamma, t-t_1+t_2)\in\mathcal{G}(\eta)$. That is, these are segments which lie in $\mathcal{G}(\eta)$ after making some bounded change to their endpoints.
\end{definition}

\begin{corollary}\label{cor:expanded weak specification}
Specification as in Proposition~\ref{specification} holds for $\mathcal{G}^M(\eta)$, with the constant $T$ depending on $M$ in addition to the parameters listed in Proposition~\ref{specification}.
\end{corollary}

\begin{proof}
This is a simple exercise using Proposition~\ref{specification} and uniform continuity of the geodesic flow. We give the idea of the proof. Let $\{(\gamma_i, t_i)\}_{i=1}^n\subset \mathcal{G}^M(\eta)$ be a collection of segments which we wish to shadow at scale $\eps$. This leads to a collection $\{(g_{s_i}\gamma_i,t_i')\}_{i=1}^n\subset\mathcal{G}(\eta)$, where $|s_i| \leq M$ and $|t_i - t_i'| \leq M$ which we can shadow at any scale as in Proposition~\ref{specification}. We choose our new shadowing scale $\delta$ so that if $d_{\GS}(\gamma,\xi) < \delta$, then $d_{\GS}(g_t\gamma,g_t\xi) < \eps$ for $t\in [-M,M]$, using uniform continuity of the flow. Any geodesic which $\delta$-shadows $\{(g_{s_i}\gamma_{t_i},t_i')\}$ must then $\eps$-shadow our desired collection $\{(\gamma_i,t_i)\}$.
\end{proof}

%

\section{$\GGG(\proportion)$ has the Bowen property}\label{sec:bowen}

In this section we establish the Bowen property (see Definition~\ref{def: Bowen property}). To do so, we analyze orbits that stay close to a good orbit segment for some time. This description will allow us to effectively bound the difference of ergodic averages along these orbits.

\begin{proposition}
For all $\proportion > 0$, for all sufficiently small $\eps > 0$ (dependent on $\proportion$), and for any $(\gamma, t)\in \GGG(\proportion)$ with $t > 2\frac{\theta_0}{2\eta}$, we have
	\begin{equation*}
	B_t(\gamma,\eps)\subset C_{2\eps,\frac{\theta_0}{2\eta}}(\gamma,t),
	\end{equation*}
where 
	\begin{equation*}
	B_t(\gamma,\eps) = \{\xi\in\GS \mid d_{\GS}(g_u\gamma,g_u\xi)<\eps \text{ for all }u\in[0,t]\}
	\end{equation*}
and 
	\begin{equation*}
	C_{2\eps,\frac{\theta_0}{2\eta}}(\gamma,t)=\left\{\xi \bigm| \exists |r| \leq 2\eps \text{ such that }g_r\xi(u) = \gamma(u) \text{ for all }u\in \left[\frac{\theta_0}{2\eta},t-\frac{\theta_0}{2\eta}\right]\right\}.
	\end{equation*}
\end{proposition}

\begin{proof}
Fix $\eta>0$, and recall Proposition~\ref{prop: distance to Con}. Now choose $\eps>0$ small enough that $s\sin(\frac{s\eta}{4}) > 2\eps e^{2(\frac{\theta_0}{2\eta} + s)}$. (Here $s$ is the parameter involved in the definition of $\lambda$, and fixed in Lemma~\ref{lem:lower semicts}.) Consider a cone around some geodesic with angle $\frac{s\eta}{4}$. By an easy computation, the ball of radius $2\eps e^{2(\frac{\theta_0}{2\eta}+s)}$ with center at distance $s$ from the cone point along the geodesic is contained in the cone (recall that $s > 0$ was chosen so that $2s < \ell_0$).

Let $(\gamma_1,t)\in\GGG(\eta)$ with $t > \frac{\theta_0}{\eta}$ be arbitrary. By Proposition~\ref{prop: distance to Con}, there exists $t_0\in [-\frac{\theta_0}{2\eta},\frac{\theta_0}{2\eta}]$ such that $\gamma_1(t_0)\in Con$ and $|\theta(\gamma_1,t_0)|-\pi \geq s\eta$. Similarly, there exists $t_1\in [-\frac{\theta_0}{2\eta},\frac{\theta_0}{2\eta}]$ such that $\gamma_1(t + t_1)\in Con$ and $|\theta(\gamma_1,t + t_1)|-\pi \geq s\eta$.
	
Now consider $\gamma_2\in B_t(\gamma_1,\eps)$. Taking $t_0$ and $t_1$ as above, by Lemmas~\ref{lem:closeness in S and GS} and \ref{lem:Lipschitz},
	\begin{equation}\label{eq: t_0}
	d_{\SSS}(\gamma_1(t_0-s),\gamma_2(t_0-s)) \leq 2d_{\GS}(g_{t_0 - s}\gamma_1,g_{t_0 - s}\gamma_2) \leq  2d_{\GS}(\gamma_1,\gamma_2) e^{2|\frac{\theta_0}{2\eta}-s|} \leq 2\eps e^{2(\frac{\theta_0}{2\eta}+s)},
	\end{equation}
and
	\begin{equation}\label{eq: t_1}
	d_{\SSS}(\gamma_1(t+t_1+s),\gamma_2(t+t_1+s)) \leq 2d_{\GS}(g_{t_1 + s}g_t\gamma_1,g_{t_1 + s}g_t\gamma_2) \leq  2d_{\GS}(g_t\gamma_1,g_t\gamma_2) e^{2|t_1+s|} \leq 2\eps e^{2(\frac{\theta_0}{2\eta}+s)}.
	\end{equation}
	
Let $\tilde{\gamma_1}$ and $\tilde{\gamma_2}$ be lifts of $\gamma_1$ and $\gamma_2$ to $G\tilde{\SSS}$ so that $d_{G\tilde{\SSS}}(\tilde{\gamma_1},\tilde{\gamma_2}) = d_{G\SSS}(\gamma_1,\gamma_2)$. Let $B_1=B(\tilde\gamma_1(t_0-s),2\eps e^{2(\frac{\theta_0}{2\eta}+s)})$ and $B_2=B(\tilde\gamma_1(t+t_1+s),2\eps e^{2(\frac{\theta_0}{2\eta}+s)})$. Then, by \eqref{eq: t_0} and \eqref{eq: t_1} and the remarks in the first paragraph of this proof, $\gamma_2$ intersects $B_1$ and $B_2$. Since $|\theta(\gamma_1,t_0)|-\pi\geq s\eta$ and $|\theta(\gamma_1,t + t_1)|-\pi\geq s\eta$, by the choice of $\eps$ and the fact that any two points in a CAT(0)-space are connected by a unique geodesic segment, $\tilde\gamma_2$ contains $\tilde\gamma_1[t_0,t+t_1]$. Moreover, since $d_{G\tilde{\SSS}}(g_{t_0+s}\tilde{\gamma_1},g_{t_0+s}\tilde{\gamma_2}) \leq \eps$ and $0\leq t_0+s\leq 2s<t$, it follows that $d_{\SSS}(\tilde\gamma_1(t_0+s),\tilde\gamma_2(t_0+s)\leq 2\eps$. Thus, there exists $r$ such that $|r|\leq 2\eps$ and $g_r\gamma_2(u)=\gamma_1(u)$ for $u\in[t_0,t+t_1]$. Since $t_0\leq \frac{\theta_0}{2\eta}$ and $t_1 \geq -\frac{\theta_0}{2\eta}$, we have completed our proof.
\end{proof}

\begin{proposition}\label{prop: Bowen property}
For all $\eps, s > 0$ and $\alpha$-H\"older continuous functions $\phi$, there exists $K > 0$ such that for all geodesic segments $(\gamma_1,t)$ with $t > 2\frac{\theta_0}{2\eta}$, given any $\gamma_2\in C_{2\eps,s}(\gamma_1,t)$, we have
$$\left|\int_{0}^{t}\phi(g_r\gamma_1) - \phi(g_r\gamma_2)\,dr\right| \leq K.$$
\end{proposition}

\begin{proof}
Let $R$ be the time-shift in the definition of $C_{2\eps,s}(\gamma_1,t)$, so that $g_R\gamma_2(r) = \gamma_1(r)$ for $r\in [s,t-s]$. We see that
\begin{align*}
\left|\int_0^t\phi(g_r\gamma_1) - \phi(g_r\gamma_2)\,dr\right| &\leq \left|\int_0^t\phi(g_r\gamma_1)\,dr - \int_{-R}^{t-R}\phi(g_r(g_R\gamma_2))\,dr\right|\\
&\leq \left|\int_{s}^{t-s}\phi(g_r\gamma_1) - \phi(g_r(g_R\gamma_2))\,dr\right|+(4s+2|R|)\|\phi\|.
\end{align*}

Since $\gamma_1 = g_R\gamma_2$ on $[s,t-s]$, by Lemma~\ref{lem:exponentially close}, we have for all $r\in[s,t-s]$
\begin{equation*}
d_{\GS}(g_r\gamma_1,g_r(g_R)\gamma_2)\leq e^{-2\min\{|r-s|,|r-(t-s)|\}}.
\end{equation*}

Thus we obtain
\begin{align*}
\left|\int_{s}^{t-s}\phi(g_r\gamma_1) - \phi(g_r(g_R\gamma_2))\,dr\right|&\leq \int_s^{t-s} C(d_{\GS}(g_r\gamma_1,g_r(g_R\gamma_2)))^\alpha\,dr\\
&\leq \int_s^{\frac{t}{2}} Ce^{-2\alpha(r-s)}\,dr+\int_{\frac{t}{2}}^{t-s}Ce^{-2\alpha((t-s)-r)}\,dr\\
&= \frac{C}{\alpha}(1-e^{-\alpha(t-2s)})\\
&\leq \frac{C}{\alpha}.
\end{align*}

As a result, since $|R|<2\eps$, we have
\begin{equation*}
\left|\int_0^t\phi(g_r\gamma_1) - \phi(g_r\gamma_2)\,dr\right|\leq \frac{C}{\alpha}+(4s+4\eps)\|\phi\|.
\end{equation*}
\end{proof}

\begin{corollary}
For all $\eta > 0$, there exists $\eps > 0$ such that $\GGG(\eta)$ has the Bowen property at scale $\eps$.
\end{corollary}

\begin{proof}
Fix $\eta > 0$. Then, choose $\eps > 0$ sufficiently small to apply the previous propositions. Then, we can take the constant for the Bowen property to be $\max\{K, 2\frac{\theta_0}{\eta}\lVert \phi\rVert\}$, where $K$ is from the previous proposition. Then, the previous proposition gives the desired bound for orbit segments of length at least $\frac{\theta_0}{\eta}$, and the triangle inequality gives the desired bound for any shorter orbit segments.
\end{proof}

%

\section{Establishing the Pressure Gap}\label{sec:pressure gap}

In this section, we prove the pressure gap condition of \cite{BCFT} for certain potentials. We then show that this pressure gap holds in the product space as well. See also the survey by Climenhaga and Thompson ~\cite[Section 14]{CT20}.

First, we prove the following theorem.
\begin{theorem}\label{thm: locally constant}
Let $\phi$ be a continuous potential that is locally constant on a neighborhood of $\Sing$. Then, $P(\Sing,\phi)<P(\phi)$.
\end{theorem}

Furthermore, we use the above theorem to note that a pressure gap also holds for functions that are nearly constant. (See Corollary~\ref{cor: nearly constant}.) For a sense of the functions covered by Theorem~\ref{thm: locally constant}, it may be helpful to think of the special case of a translation surface. There are infinitely many cylinders in such an $\SSS$, and the geodesics circling different cylinders are in different connected components of $\Sing$, so there is significant flexibility in building a function that satisfies Theorem~\ref{thm: locally constant} on $\Sing$ itself, let alone on the complement of its neighborhood.

Our argument for Theorem~\ref{thm: locally constant} closely follows that in \S8 of \cite{BCFT}. The different geometry in our situation calls for somewhat different arguments in Proposition~\ref{prop:regular shadow} and Lemma~\ref{lem:mult bound}, which we present here in full. After these are proved, the argument hews closely to \cite{BCFT}. We present the main steps of the argument, filling in the details where a modification is necessary for the present situation.

For any $\eta>0$, we let
\[ \Reg(\eta) = \{\gamma \mid \lambda(\gamma)\geq \eta\}.\]

We need a pair of Lemmas in this section.

\begin{lemma}\label{lem:extend sing}
Let $c$ be any singular geodesic segment. That is, $c$ is a geodesic segment such that the turning angle at any cone points it encounters is always $\pm\pi$. Then $c$ can be extended to a complete geodesic $\gamma \in \Sing$.
\end{lemma}

\begin{proof}
The extension is accomplished by following the geodesic trajectory established by $c$ and, whenever a cone point is encountered, continuing the extension so that a turning angle of $\pi$ or $-\pi$ is made.
\end{proof}

Let $\partial_\infty\tilde \SSS$ be the boundary at infinity of $\tilde \SSS$, equipped with the usual cone topology (see, e.g., \cite[\S II.8]{bh}). Since $\SSS$ is a surface, $\partial_\infty\tilde \SSS$ is a circle. Using this identification, we can speak of a path in $\partial_\infty\tilde \SSS$ as being monotonic if it always moves in a clockwise or counterclockwise direction.

The following Lemma leverages this structure to provide a way to continuously move a geodesic in $G\tilde\SSS$.

\begin{lemma}\label{lem:cts pivot}
Let $\tilde\gamma \in G\tilde\SSS$ with $\tilde \gamma(t_0)=\xi\in\widetilde{Con}$. Let $\zeta_v$ be a continuous and monotonic path in $\partial_\infty\tilde\SSS$ with $\zeta_0=\tilde\gamma(+\infty)$ such that for all $v$, the ray connecting $\xi$ with $\zeta_v$ can be concatenated with $\tilde\gamma(-\infty,t_0)$ to form a geodesic $\tilde\gamma_v$. Then $\{\gamma_v\}$ is a continuous path of geodesics in $G\tilde\SSS$ with $d_{G\tilde\SSS}(\tilde\gamma, \tilde\gamma_v)$ non-decreasing in $|v|$.
\end{lemma}

\begin{proof}
First, that $\xi$ and $\zeta_v$ can be connected with a unique geodesic ray is a standard fact about CAT(0) spaces (\cite[\S II.8, Prop 8.2]{bh}). For continuity of $\tilde\gamma_v$, we claim that if $v\to v_0$, $d_{\tilde S}(\tilde\gamma_v(t), \tilde\gamma_{v_0}(t))\to 0$ uniformly on any $[t_0,T]$. This together with the formula for $d_{G\tilde \SSS}$ will show that $d_{G\tilde \SSS}(\gamma_v,\gamma_{v_0})\to 0$. To verify the claim, fix $T>t_0$ and $\eps>0$ and recall that in the cone topology on $\partial_\infty\tilde \SSS$,
\[ U(\tilde\gamma_{v_0}, T, \eps) := \{ \zeta\in \partial_\infty\tilde \SSS : d_{\tilde \SSS} (c(T), \tilde\gamma_{v_0}(T))< \eps \mbox{ where } c \mbox{ is the geodesic ray from } \xi \mbox{ to } \zeta\}\]
is a basic open set around $\zeta_{v_0} = \tilde \gamma_{v_0}(+\infty)$ (\cite[\S II.8]{bh}). Therefore, for $v$ sufficiently close to $v_0$, $\zeta_v \in U(\tilde\gamma_{v_0}, T ,\eps).$ But the ray $c$ from $\xi$ to $\zeta_v$ is precisely $\tilde\gamma_v|_{[t_0,+\infty)}$. Thus $d_{\tilde \SSS}(\tilde \gamma_v(T), \tilde \gamma_{v_0}(T))<\eps$. Since the distance between two geodesics is a convex function of the parameter (\cite[\S II.2]{bh}) and $d_{\tilde S}(\tilde \gamma_v(t_0), \tilde \gamma_{v_0}(t_0))=0$, for all $t\in [t_0, T]$ we have $d_{\tilde \SSS}(\tilde \gamma_v(t), \tilde \gamma_{v_0}(t))<\eps$ and hence have the desired uniform convergence.

For all $t\leq t_0$, $d_{\tilde\SSS}(\tilde\gamma(t),\tilde\gamma_v(t))=0$. We claim that for $t>t_0$, $d_{\tilde S}(\tilde\gamma(t),\tilde\gamma_v(t))$ is non-decreasing in $|v|$. Together with the formula for $d_{G\tilde \SSS}$, this will provide the result.

Fix some $v^*\neq 0$; without loss of generality, we can assume $v^*>0$. Since $\zeta_v$ is a monotonic path on $\partial_\infty\tilde \SSS$, $v \mapsto \tilde\gamma_v(t)$ sweeps out an arc on the circle of radius $t-t_0$ centered at $\xi$ monotonically (though not necessarily strictly monotonically). We want to show that for $v>v^*$, $d_{\tilde S}(\tilde\gamma(t),\tilde\gamma_v(t)) \geq d_{\tilde S}(\tilde\gamma(t),\tilde\gamma_{v^*}(t))$. This will be trivially true if for all $v>v^*$, $\tilde \gamma_v(t) = \tilde \gamma_{v^*}(t)$, so we can assume this is not the case.

Consider the path swept out by $v \mapsto \tilde \gamma_v(t)$. Near the point $\tilde\gamma_{v^*}(t)$ this path consists of arcs of two Euclidean circles meeting at $\tilde \gamma_{v^*}(t)$. To each side of $\tilde \gamma_{v^*}(t)$, the arc belongs to a circle centered at the cone point on $[\xi, \tilde \gamma_{v^*}(t)]\setminus\{ \tilde \gamma_{v^*}(t)\}$ closest to $\tilde\gamma_{v^*}(t)$ among those cone points where $[\xi,\tilde\gamma_{v^*}(t)]$ makes angle greater than $\pi$ on the given side of $[\xi,\tilde\gamma_{v^*}(t)]$. Therefore in the space of directions at $\tilde \gamma_{v^*}(t)$ (this will be the tangent space at $\tilde \gamma_{v^*}(t)$ unless $\tilde \gamma_{v^*}(t)$ happens to be a cone point) we have well-defined vectors pointing along these arcs. Furthermore, since these are arcs of Euclidean circles, the angles between these two vectors and a vector pointing radially along $[\tilde \gamma_{v^*}(t), \xi]$ are both $\frac{\pi}{2}$. Let $W^+$ and $W^-$ be vectors in the space of directions at $\tilde \gamma_{v^*}(t)$ pointing along the arc swept out by $v \mapsto \tilde \gamma_v(t)$ with $W^+$ pointing in the direction swept out as $v$ increases past $v^*$ and $W^-$ in the direction swept out as $v$ decreases from $v^*$. (Note that $v\mapsto \tilde \gamma_v(t)$ may be constant in $v$ for $v$ near $v^*$ due to cone points $\tilde\gamma_{v^*}$ encounters at times greater than $t$. The vectors $W^\pm$ are tangent to a reparametrization of this curve by arc-length, for example.) Similarly, let $H^\pm$ be the vectors in the space of directions at $\tilde \gamma_{v^*}(t)$ pointing along the circle of radius $d_{\tilde \SSS}(\tilde \gamma(t), \tilde \gamma_{v^*}(t))$ centered at $\tilde \gamma(t)$. Let $V_1$ be the initial tangent vector for the geodesic segment from $\tilde \gamma_{v^*}(t)$ to $\xi$ and let $V_2$ be the initial tangent vector for the geodesic segment from $\tilde \gamma_{v^*}(t)$ to $\tilde \gamma(t)$. By the CAT(0) condition and using a comparison triangle for the triangle with vertices $\xi, \tilde \gamma(t)$, and $\tilde \gamma_{v^*}(t)$ it is easy to check that the angle between $V_1$ and $V_2$ is in $[0, \frac{\pi}{2})$. The angles between $W^\pm$ and $V_1$ and between $H^\pm$ and $V_2$ are all $\frac{\pi}{2}$ as these are angles between a circle and one of its radial segments. (See Figure \ref{fig:non-decreasing}.)

\begin{figure}[h]
\centering
\begin{tikzpicture}[scale=1.3]

\draw[thick] (-3, 0) --  (3,0) ;
\draw[thick] (-2,0) -- (3,3.5);
\draw[thick] (3, 0) --  (2.08,2.88) ;

\fill [black] (-2, 0) circle(.06) ;
\node[black] at (-2, -.3) {$\xi$} ;

\draw[thick, dashed, red] (3,0) arc (0:70:5cm);
\draw[thick, dashed, blue] (1,2.28) arc (130:85:3cm);

\fill [black] (3, 0) circle(.06) ;
\node[black] at (3, -.3) {$\tilde\gamma(t)$} ;

\node[black] at (2.3, 4) {$\tilde\gamma_{v^*}(t)$} ;

\draw[ultra thick, ->] (2.09, 2.86) -- (1.3,2.31) ;
\draw[ultra thick, ->] (2.09, 2.86) -- (2.36,2) ;
\draw[red, ultra thick, ->] (2.09, 2.86) -- (1.57,3.6) ;
\draw[red, ultra thick, ->] (2.09, 2.86) -- (2.61,2.12) ;
\draw[blue, ultra thick, ->] (2.09, 2.86) -- (3,3.1) ;
\draw[blue, ultra thick, ->] (2.09, 2.86) -- (1.18,2.62) ;

\fill [black] (2.09, 2.86) circle(.06) ;

\node[black] at (1.5, 2) {$V_1$} ;
\node[black] at (2.2, 1.8) {$V_2$} ;
\node[red] at (1.2, 3.4) {$W^+$} ;
\node[red] at (3, 2) {$W^-$} ;
\node[blue] at (3.3, 3.3) {$H^+$} ;
\node[blue] at (1, 2.8) {$H^-$} ;

\draw[thick, ->] (2.3, 3.86) -- (2.1,3) ;

\end{tikzpicture}
\caption{The proof that $d_{\tilde \SSS}(\tilde \gamma(t), \tilde\gamma_v(t))$ is non-decreasing in $|v|$.}\label{fig:non-decreasing}
\end{figure}
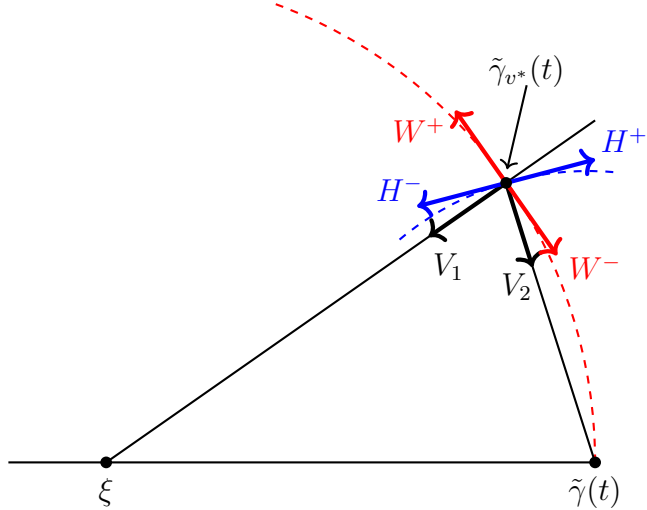

The segment $[\tilde\gamma_{v^*}(t), \tilde \gamma(t)]$ lies in the convex hull of $\tilde \gamma$ and $\tilde \gamma_{v^*}$. By the CAT(0) condition, it is within the ball of radius $t-t_0$ centered at $\xi$. So $V_2$, which points along $[\tilde\gamma_{v^*} (t), \tilde \gamma(t)]$, is between $V_1$ and $W^-$ in the space of directions at $\tilde\gamma_{v^*}(t)$. More precisely, the space of directions at $\tilde \gamma_{v^*}(t)$ is a circle with total length equal to the total angle at $\tilde \gamma_{v^*}(t)$. $V_2$ is between $V_1$ and $W^-$ in the sense that it lies within the angle-$\frac{\pi}{2}$ arc of directions connecting $V_1$ and $W^-$ in the space of directions. Thus, the angle between $V_2$ and $W^-$ is less than or equal to $\frac{\pi}{2}$ and so the angle between $V_2$ and $W^+$ is at least $\frac{\pi}{2}$.

If the angle between $V_2$ and $W^+$ is $\frac{\pi}{2}$, then the geodesic segment $[\tilde \gamma(t), \tilde \gamma_{v^*}(t)]$ must run through $\xi$ and then for $v>v^*$, $d_{\tilde \SSS}(\tilde \gamma(t), \tilde \gamma_v(t)) = 2(t-t_0) = d_{\tilde \SSS}(\tilde \gamma(t), \tilde \gamma_{v^*}(t))$. If the angle is strictly less than $\frac{\pi}{2}$ then in the space of directions, $W^+$ is separated from $V_2$ by $H^\pm$. This means that as the path $v\mapsto \tilde \gamma_v(t)$ leaves the point $\tilde \gamma_{v^*}(t)$ with $v$ increasing, it must move -- at least initially -- to the outside of the circle of radius $d_{\tilde \SSS}(\tilde \gamma(t), \tilde \gamma_{v^*}(t))$ centered at $\tilde \gamma(t)$. In particular, $d_{\tilde \SSS}(\tilde \gamma(t), \tilde \gamma_{v^*}(t))$ is locally monotonically increasing near $v^*$. As $v^*$ was arbitrary (among $v$ such that $\tilde \gamma_v$ give geodesic extensions of $\tilde \gamma(-\infty, t_0)$), and the path $v\mapsto \tilde \gamma_v(t)$ is connected, this completes our proof of the claim and the lemma.
\end{proof}

The first step in the dynamical argument for a pressure gap is the following technical Proposition, which allows us to find a regular geodesic which is close to any connected component of the $\delta$-neighborhood of the singular set.

\begin{proposition}\label{prop:regular shadow}
Let $\delta>0$  and $0<\eta<\frac{\eta_0}{2s}$ be given, where $\eta_0$ is defined in Lemma~\ref{lem:no big flat}(b). Then there exists $L>0$ and a family of maps $\Pi_t:\Sing \to \Reg(\eta)$ such that for all $t>3L$ and for all $\gamma \in \Sing$, if we write $c = \Pi_t(\gamma)$ then the following are true:
\begin{itemize}
	\item[(a)] $c, g_{t+t'}c \in \Reg(\eta)$ for some $|t'|<4d_0$;
	\item[(b)] $d_{GS}(g_r c, \Sing) < \delta$ for all $r\in [L, t-L]$;
	\item[(c)] for all $r\in [L, t-L]$, $g_r c$ and $\gamma$ lie in the same connected component of $B(\Sing, \delta)$, the $\delta$-neighborhood of $\Sing$.
\end{itemize}
Furthermore, $c(0), c(t+t')\in Con$, any $c\in \Pi_t(\Sing)$ is entirely determined (among the geodesics in $\Pi_t(\Sing)$) by the segment $c[0,t+t']$, and $d_{\tilde \SSS}(\gamma(0), c(0)), d_{\tilde \SSS}(\gamma(t),c(t+t'))<2d_0$ where $d_0$ is as in Lemma~\ref{lem:no big flat}(a).
\end{proposition}

\begin{remark*}
The above proposition should be compared with \cite[Theorem 8.1]{BCFT}, although we have made two slight adjustments for our situation. First, we cannot guarantee that $g_tc\in \Reg(\eta)$, but only that $g_{t+t'}c\in \Reg(\eta)$ with uniform control on $|t'|$. Second, we prove our result for all $t>3L$, instead of $2L$. These result in trivial changes to subsequent estimates in \cite{BCFT}'s argument.
\end{remark*}

\begin{proof}[P\textbf{roof of Proposition~\ref{prop:regular shadow}}]

We begin with a geometric preliminary.

\begin{itemize}
	\item[(A)] Suppose that $\tilde c_1$ and $\tilde c_2$ are geodesic rays in $\tilde \SSS$ with $\tilde c_1(0) = \tilde c_2(0)$ and $d_{\tilde \SSS}(\tilde c_1(l), \tilde c_2(l)) \leq 3d_0$. The distance between geodesic rays is a convex function in a CAT(0) space, so $d_{\tilde \SSS}(\tilde c_1(r), \tilde c_2(r)) \leq \frac{3d_0}{l}r$ for all $r\in [0,l]$. Therefore, to ensure that $d_{\tilde \SSS}(\tilde c_1(r), \tilde c_2(r)) < \frac{\delta}{2}$ for all $r\in[0, 2T]$ it is sufficient to have $\frac{3d_0}{l}2T<\frac{\delta}{2}$, or $l>\frac{12 d_0 T}{\delta}$.
\end{itemize}

We now begin the proof in earnest. Let $\delta>0$ and $0<\eta < \frac{\eta_0}{2s}$ be given. Let $T(\delta)$ be as in Lemma~\ref{lem:shadow in S}. Let
\[ L = \max\left\{ d_0, \frac{8d_0}{\eta_0}, 2T(\delta), \frac{12 d_0 T(\delta)}{\delta}\right\};\]
we will highlight the need for each condition on $L$ as we come to it in the proof. Let $t>3L$ and let $\gamma\in \Sing$.

As usual, we work in $\tilde \SSS$. Let $R$ be the maximal, isometrically embedded Euclidean rectangle with $\tilde \gamma([L, t-L])$ as one side, containing no cone points in its interior, and to the right side of $\tilde \gamma$, with respect to its orientation. (Throughout this proof, refer to Figure \ref{fig:regular shadow}. For ease of exposition we will often refer to the orientation as depicted in that figure in this proof.) Note that if $\tilde \gamma([L, t-L])$ contains any cone points with an angle $>\pi$ on the right side of $\tilde \gamma$, then $R$ has height zero. That $t>3L$ implies $R$ has positive width. By maximality of $R$, there must be cone points on the boundary of $R$, specifically on the bottom side of $R$, as oriented in Figure~\ref{fig:regular shadow}. Let $\xi_1$ be the cone point closest to $\tilde\gamma(0)$ and $\xi_2$ be the cone point closest to $\tilde \gamma(t)$ on the bottom side of $R$.

Using Lemma \ref{lem:extend sing}, extend the bottom side of $R$ to a singular geodesic $\tilde \gamma'$ which turns with angle $\pi$ on the $\tilde \gamma$ side\footnote{That is, measured from within the connected component of $\tilde\SSS\setminus \tilde\gamma'$ containing $\tilde \gamma$, the incoming and outgoing directions of $\tilde\gamma'$ make angle $\pi$ at any cone point.} any time it encounters a cone point. (If $R$ has height zero, let $\tilde\gamma'=\tilde \gamma$.) Parametrize $\tilde \gamma'$ so that $\tilde \gamma'(L)$ is the lower-left corner of $R$ (and hence $\tilde\gamma'(t-L)$ is the lower-right corner).

Construct geodesic segments of length $d_0$, starting at the points $\tilde \gamma'(-\frac{d_0}{2})$ and $\tilde \gamma'(\frac{d_0}{2})$, ending below $\tilde \gamma'$, and perpendicular to $\tilde \gamma'$ in the sense that for each segment, both angles between it and $\tilde\gamma'$ are $\geq \frac{\pi}{2}$. Connect the endpoints of these segments with a geodesic segment, forming a quadrilateral which we denote by $Q_1$. Construct a similar quadrilateral $Q_2$ based on $\tilde \gamma'$ around $\tilde \gamma'(t)$ on the same side as $Q_1$. Any point in $Q_1$ (resp. $Q_2$) can be reached from $\tilde \gamma'(0)$ (resp. $\tilde \gamma'(t)$) via a path along $\tilde \gamma'$ of length $\leq \frac{d_0}{2}$ followed by a perpendicular segment of length $\leq d_0$. Therefore, for any $\zeta\in Q_1$, $d_{\tilde \SSS}(\tilde \gamma'(0), \zeta)\leq\frac{3}{2}d_0< 2d_0$ (the analogous bound holds for $Q_2$) and the diameter of $Q_i$ is bounded by $3d_0$. Our choice of $L\geq d_0$ implies that $\tilde \gamma'(L)$ and $\tilde \gamma'(t-L)$ are not in the quadrilaterals.

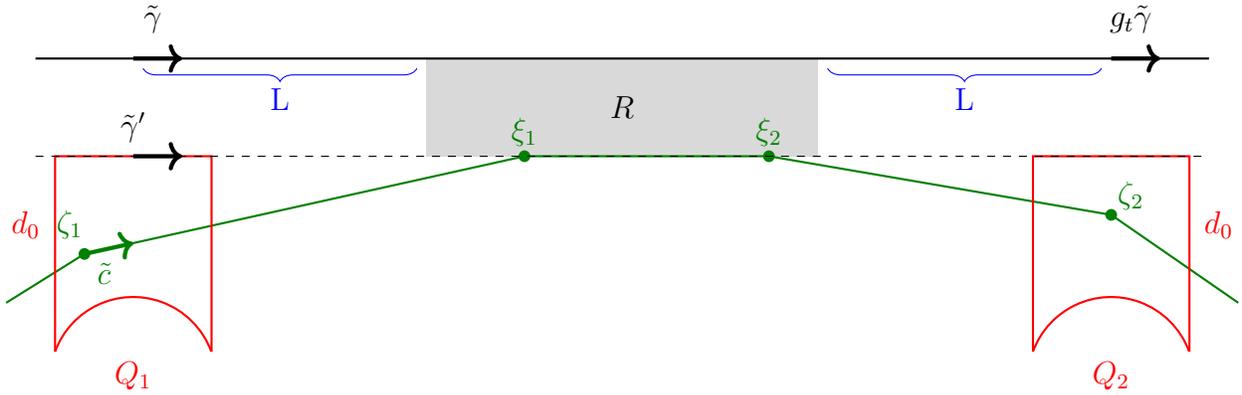
\begin{figure}[h]
\centering
\begin{tikzpicture}[scale=1.3] 

\fill[gray!30!white] (-2,3) -- (2,3) -- (2,2) -- (-2,2) -- (-2,3) ;

\draw[thick, green!50!black] (-6.3, .5) --  (-5.5,1) -- (-1,2) -- (1.5,2) -- (5,1.4) -- (6.3, .5) ;

\draw[thick] (-6,3) -- (6,3);

\draw[thick, red] (-5.8, 0) -- (-5.8, 2) -- (-4.2, 2) -- (-4.2, 0) ;
\draw[thick, red] (5.8, 0) -- (5.8, 2) -- (4.2, 2) -- (4.2, 0) ;

\draw[ultra thick, ->] (-5,3) -- (-4.5,3) ;
\draw[ultra thick, ->] (5,3) -- (5.5,3) ;

\node at (-4.8,3.4){$\tilde \gamma$};
\node at (5.2,3.4){$g_t\tilde \gamma$};

\draw[thick, red] (-5.8,0) arc (160:20:.85cm);
\draw[thick, red] (4.2,0) arc (160:20:.85cm);

\node[red] at (-6.1,1.3) {$d_0$} ; 
\node[red] at (6.1,1.3) {$d_0$} ;

\draw [ultra thick, green!50!black, ->] (-5.5,1) -- (-5,1.11) ;

\fill [green!50!black] (-5.5, 1) circle(.06) ;
\fill [green!50!black] (5, 1.4) circle(.06) ;
\fill [green!50!black] (-1, 2) circle(.06) ;
\fill [green!50!black] (1.5, 2) circle(.06) ;

\node[green!50!black] at (-5.65, 1.3) {$\zeta_1$} ;
\node[green!50!black] at (5.2, 1.6) {$\zeta_2$} ;
\node[green!50!black] at (-1, 2.25) {$\xi_1$} ;
\node[green!50!black] at (1.5, 2.25) {$\xi_2$} ;

\node [green!50!black] at (-5.3, .8) {$\tilde c$} ;

\draw [blue, decorate,decoration={brace,amplitude=5pt,mirror}]
  (-4.9,2.9) -- (-2.1,2.9) node[midway,yshift=-1em]{L};
\draw [blue, decorate,decoration={brace,amplitude=5pt,mirror}]
  (2.1,2.9) -- (4.9,2.9) node[midway,yshift=-1em]{L};

\draw [dashed] (-6,2) -- (6,2) ;

\draw [ultra thick, ->] (-5, 2) -- (-4.5,2) ;
\node at (-5,2.3) {$\tilde \gamma'$} ;

\node at (0,2.5) {$R$} ;

\node[red] at (-5,-.25) {$Q_1$} ;
\node[red] at (5,-.25) {$Q_2$} ;

\end{tikzpicture}
\caption{The construction of $c=\Pi_t(\gamma)$ in Proposition~\ref{prop:regular shadow}.}\label{fig:regular shadow}
\end{figure}

By their construction using $d_0$ from Lemma~\ref{lem:no big flat}, the quadrilaterals  $Q_1$ and $Q_2$ must contain cone points. Let $\zeta_1$ be a cone point in $Q_1$ and $\zeta_2$ a cone point in $Q_2$. Let $\hat t = d_{\tilde \SSS}(\zeta_1, \zeta_2)$. Extend the geodesic segment $[\zeta_1,\zeta_2]$ to a geodesic $\tilde c$, parameterized so that $\tilde c(0)=\zeta_1$ and $\tilde c(\hat t) = \zeta_2$, with turning angles equal to exactly half of the total angle at each cone point $\zeta_1$, $\zeta_2$, and any cone points encountered over times $(-\infty,0] \cup [\hat t,\infty)$. Note that this condition implies that $c$ is determined entirely by the segment $[\zeta_1, \zeta_2]$.  Then, $\tilde c\in \Reg(\eta)$. An alternate path from $\zeta_1$ to $\zeta_2$ is to travel $\zeta_1\to \tilde \gamma'(0) \to \tilde \gamma'(t)\to \zeta_2$ which has length $< 4d_0+t$. Thus, $\hat t < t+ 4d_0$. Reversing the roles of $\tilde c$ and $\tilde \gamma'$ also shows $t<\hat t+4d_0$, so $\hat t = t+t'$ with $|t'|<4d_0$. Then, $g_{t+t'}\tilde c \in \Reg(\eta)$ as desired.

We claim that $[\xi_1,\xi_2] \subseteq \tilde c \cap R$. Consider the geodesic segments $[\zeta_1,\xi_1]$, $[\xi_1,\xi_2]$, and $[\xi_2, \zeta_2]$. The triangle formed by $\tilde \gamma'(0)$, $\zeta_1$ and $\xi_1$ has $d_{\tilde\SSS}(\tilde\gamma'(0),\xi_1)\geq L$ and as noted above, $d_{\tilde\SSS}(\tilde\gamma'(0),\zeta_1)<2d_0$. Using the CAT(0) property and an easy Euclidean geometry calculation, the angle between $[\tilde\gamma'(0), \xi_1]$ and $[\zeta_1,\xi_1]$ at $\xi_1$ is less than $\frac{4d_0}{L}$. Our assumption that $L\geq \frac{8d_0}{\eta_0}$ ensures that this angle is less than $\frac{\eta_0}{2}$. An analogous argument bounds the angle between $[\xi_2,\tilde\gamma'(t)]$ and $[\xi_2,\zeta_2]$. By Lemma \ref{lem:no big flat} there is excess angle at least $\eta_0$ at $\xi_1$ and $\xi_2$. At $\xi_1$ (and similarly at $\xi_2$, even if $\xi_1=\xi_2$) the angle our concatenation of segments makes on the side towards $\tilde \gamma$ is at least the angle $\tilde \gamma'$ makes on that side, which by construction is $\pi$. On the side away from $\tilde \gamma$, the angle our concatenation makes is at least $\mathcal{L}(\xi_1) - \pi -\eta_0>\pi$. The concatenation of $[\zeta_1,\xi_1]$, $[\xi_1,\xi_2]$, and $[\xi_2, \zeta_2]$ is therefore a geodesic segment, and hence it must be a subsegment of $\tilde c$, proving the claim.

We now need to show, using our choice of $L$, that $d_{G\tilde \SSS}(g_{r^*}\tilde c, \Sing)<\delta$ for all $r^*\in[L, t-L]$. We do this by showing that for each such $r^*$ there is a geodesic $\tilde \gamma' \in \Sing$ such that $d_{\tilde \SSS}(\tilde \gamma'(r), \tilde c(r))<\frac{\delta}{2}$ for all $r\in [r^*-T(\delta), r^*+T(\delta)]$ and then invoking Lemma~\ref{lem:shadow in S}.


Let $\tilde\gamma'_0$ be the reparameterization of $\tilde \gamma'$ so that $\tilde c(r) = \tilde \gamma'_0(r)$ whenever $\tilde c(r)\in R$. Let $[r_1, r_2] = \{ r : \tilde c(r) = \tilde \gamma'_0(r) \in \tilde c \cap R \}$. (Figure \ref{fig:regular shadow} depicts a situation where $\tilde c(r_1)=\xi_1$ and $\tilde c(r_2)=\xi_2$.) For any $r\in [r_1, r_2]$, consider the geodesic rays $\tilde c(-\infty, r)$ and $\tilde \gamma'_0(-\infty,r)$. They share the point $\tilde c(r) = \tilde \gamma'_0(r)$ and at some distance $\geq L \geq \frac{12 d_0 T(\delta)}{\delta}$ are both in $Q_1$ and hence $\leq3d_0$ apart (with respect to $d_{\tilde \SSS}$). Therefore, by (A) at the start of this proof, for all $r\in[r_1-2T(\delta), r_2]$, $d_{\tilde \SSS}(\tilde c(r), \tilde \gamma'_0(r))<\frac{\delta}{2}$. Applying the same argument to the rays $\tilde c(r,\infty)$ and $\tilde \gamma'_0(r,\infty)$, shows $d_{\tilde \SSS}(\tilde c(r), \tilde \gamma'_0(r))<\frac{\delta}{2}$ for all $r\in [r_1, r_2+2T(\delta)]$. As $\tilde \gamma'_0\in \Sing$, by Lemma~\ref{lem:shadow in S}, $d_{G\tilde \SSS}(g_{r^*}\tilde c, \Sing)<\delta$ for all $r^*\in [r_1-T(\delta), r_2+T(\delta)].$

If this covers all times in $[L, t-L]$, we are done with this part of the proof. If not, we continue as follows. Assuming $r_1-T(\delta)>L$, consider the geodesic segment $[\zeta_1, \xi_1]$. Let $[\xi_1^-, \xi_1]$ be the maximal subsegment of $[\zeta_1, \xi_1]$ containing no cone points in its interior. Extend $[\xi_1^-, \xi_1]$ to a geodesic $\tilde \gamma'_{-1}$ in $\Sing$ lying between $\tilde c$ and $\tilde \gamma'_0$, parametrized so that $\tilde \gamma'_{-1}(r_1) = \xi_1 = \tilde c(r_1)$. First note that over the interval $[r_1-2T(\delta), r_1]$, $\tilde \gamma'_{-1}(r)$ is at least as close to $\tilde \gamma'_0(r)$ as $\tilde c(r)$ is, and by our work above, this distance is bounded above by $\frac{\delta}{2}$. By Lemma~\ref{lem:shadow in S}, $d_{G\tilde \SSS}(g_{r_1-T(\delta)}\tilde \gamma'_0, g_{r_1-T(\delta)}\tilde \gamma'_{-1} )<\delta$. Second, we can argue regarding $\tilde c$ and $\tilde \gamma'_{-1}$ exactly as we did regarding $\tilde c$ and $\tilde \gamma'_0$. They form rays with a common endpoint which after some distance $>L$ are still within $3d_0$ of each other, which as noted above allows us to show they are $\delta$ close in $d_{G\tilde S}$ for an interval of time below $r_1$. This interval will either extend to $L$ as desired, or will end at some $r_0-T(\delta)$ where $\tilde c$ and $\tilde \gamma'_{-1}$ branch apart at a cone point. We then repeat our argument at that cone point, finding $\tilde \gamma'_{-2}\in\Sing$ shadowing $\tilde c$ further, and so on, until we have reached time $L$. Exactly the same argument applies beyond $\xi_2$, constructing $\tilde \gamma'_1, \tilde \gamma'_2, \ldots \in \Sing$ as necessary to shadow $\tilde c$ in $d_{G\tilde\SSS}$ until time $t-L$.

It remains to establish that for all $r\in[L, t-L]$, $g_r\tilde c$ and $\tilde \gamma$ lie in the same connected component of $B(\Sing,\delta)$. We do this by showing that one can get from $\tilde \gamma$ to $g_r\tilde c$ by a series of `moves,' each of which can be realized by a continuous path in $B(\Sing,\delta)$.

\

\noindent\textbf{Move 1: geodesic flow} 

If $\tilde \gamma \in \Sing$ then for all $r$, $g_r\tilde \gamma\in \Sing$ with the flow itself providing a continuous path between the two, so $\tilde \gamma$ and $g_r\tilde \gamma$ are both in the same connected component of $\Sing$ itself, and hence of $B(\Sing,\delta)$. 

\

\noindent\textbf{Move 2: `pivot'} 

Suppose $\tilde \gamma_i, \tilde \gamma_{i+1} \in \Sing$ with $\tilde \gamma_i(0)=\tilde \gamma_{i+1}(0) = \xi\in \widetilde{Con}$, $d_{G\tilde \SSS}(\tilde \gamma_i, \tilde \gamma_{i+1})<\delta$, and suppose that the angle between $\tilde\gamma_i$ and $\tilde \gamma_{i+1}$ at $\xi$ is less than $\mathcal{L}(\xi)-2\pi$. Note that any geodesic ray starting from $\xi$ which lies between $\tilde \gamma_i(-\infty,0)$ and $\tilde \gamma_{i+1}(-\infty,0)$ can be concatenated with $\tilde\gamma_i(0,+\infty)$ to form a geodesic. Similarly, any ray between $\tilde \gamma_i(0,+\infty)$ and $\tilde \gamma_{i+1}(0,+\infty)$ can be concatenated with $\tilde \gamma_{i+1}(-\infty,0)$ to form a geodesic.

Let $\tilde \gamma_{i+1}\cdot \tilde \gamma_i$ be the concatenation of $\tilde \gamma_{i+1}(-\infty,0]$ with $\tilde \gamma_i [0, +\infty)$. Note that $d_{G\tilde \SSS}(\tilde \gamma_i, \tilde \gamma_{i+1}\cdot\tilde \gamma_i)$ and $d_{G\tilde \SSS}(\tilde \gamma_{i+1}\cdot\tilde \gamma_i, \tilde \gamma_{i+1})$ are both less than $d_{G\tilde \SSS}(\tilde \gamma_i, \tilde \gamma_{i+1})$ and hence less than $\delta.$ Indeed, the integrals computing $d_{G\tilde \SSS}(\tilde \gamma_i, \tilde \gamma_{i+1}\cdot\tilde \gamma_i)$ and $d_{G\tilde \SSS}(\tilde \gamma_{i+1}\cdot\tilde \gamma_i, \tilde \gamma_{i+1})$ will each match the integral to compute $d_{G\tilde \SSS}(\tilde \gamma_i, \tilde \gamma_{i+1})$ on one side of $t=0$, and will replace the integral on the other side of $t=0$ by zero, if anything decreasing the distance.

We `pivot' from $\tilde \gamma_i$ to $\tilde \gamma_{i+1}$ in two steps. First, let $\zeta_v$ be a continuous and monotonic path in $\partial_\infty\tilde\SSS$ from $\zeta_0 = \tilde \gamma_i(-\infty)$ to $\zeta_1 = \tilde \gamma_{i+1}(-\infty)$. Apply Lemma~\ref{lem:cts pivot} to get a continuous path $v\mapsto \tilde c_v$ from $\tilde \gamma_i$ to $\tilde \gamma_{i+1} \cdot \tilde \gamma_i$ such that for all $v$, $d_{G\tilde \SSS}(\tilde\gamma_i, \tilde c_v) \leq d_{G\tilde \SSS}(\tilde \gamma_i, \tilde \gamma_{i+1}\cdot \tilde \gamma_i)<\delta.$ Second, let $\zeta'_v$ be a continuous and monotonic path from $\tilde \gamma_i(+\infty)$ to $\tilde\gamma_{i+1}(+\infty)$ and apply Lemma~\ref{lem:cts pivot} to get a continuous path $v\mapsto \tilde c'_v$ from $\tilde \gamma_{i+1}\cdot\tilde \gamma_i$ to $\tilde \gamma_{i+1}$. Again, for all $v$, $d_{G\tilde \SSS}(\tilde c'_v, \tilde \gamma_{i+1})< d_{G\tilde \SSS}(\tilde \gamma_{i+1}\cdot\tilde \gamma_i, \tilde \gamma_{i+1})<\delta$; this time we apply the distance non-increasing property obtained in Lemma \ref{lem:cts pivot} to the reverse of the path $v\mapsto \tilde c'_v$, which continuously moves from $\tilde \gamma_{i+1}$ to $\tilde \gamma_{i+1}\cdot \tilde \gamma_i$. Overall, we have a path of geodesics that remains in $B(\Sing,\delta)$ throughout.

\

\noindent\textbf{Move 3: `slide'} 

Suppose that $R$ is an isometrically embedded Euclidean rectangle in $\tilde \SSS$. (Note that this implies $R$ contains no cone points in its interior.) Let $\tilde \gamma, \tilde \gamma' \in \Sing$ be geodesics which extend the top and bottom sides of $R$, respectively, with $d_{G\tilde \SSS}(\tilde\gamma, \tilde \gamma')<\delta$. Let $\{c_v\}$ be a continuous path of horizontal (i.e., parallel to $\tilde \gamma$ and $\tilde \gamma'$ within $R$) geodesic segments connecting the two sides of $R$, which move monotonically downward through $R$, with $c_0=\tilde \gamma\cap R$ and $c_1=\tilde \gamma'\cap R$.

For each $v$, let $\tilde \gamma^u_v$ be the `uppermost' geodesic extension of $c_v$, that is, the extension which turns with angle $\pi$ on the $\tilde \gamma$-side at any cone point it hits. Let $\tilde \gamma^l_v$ be the `lowermost' geodesic extension of $c_v$, that is, the extension which turns with angle $\pi$ on the $\tilde \gamma'$-side at any cone point it hits. Since the distance between geodesics is a convex function and since $c_v$ is parallel to $\tilde \gamma$ and $\tilde \gamma'$ over $R$, both $\tilde\gamma^u_v$ and $\tilde\gamma^l_v$ lie between $\tilde \gamma$ and $\tilde \gamma'$.

If $\tilde\gamma_v^u=\tilde \gamma_v^l$, set $\tilde \gamma_v=\tilde\gamma_v^u=\tilde \gamma_v^l$. This happens if and only if $\tilde \gamma_v$ hits no cone points. Since there are countably many cone points in $\tilde \SSS$, there is a countable set $\{ v_n \}\subset [0,1]$ for which $\tilde\gamma_{v_n}^u \neq \tilde \gamma_{v_n}^l$. Let $\{ I_n\}$ be a corresponding collection of closed real intervals with $\sum|I_n|=1$. Cut $[0,1]$ at each $v_n$ and glue in the interval $I_n$, resulting in an interval of length 2. Adjust the subscripts where $\tilde\gamma_v$ has already been defined accordingly. For each $n$, if $I_n = [a_n,b_n]$ set $\tilde \gamma_{a_n}=\tilde \gamma^u_{v_n}$ and $\tilde \gamma_{b_n}=\tilde \gamma^l_{v_n}$. For all $v\in I_n$, use Lemma~\ref{lem:cts pivot} to fill in a continuous path $v \mapsto \tilde \gamma_v$ from $\tilde \gamma_{a_n}$ to $\tilde \gamma_{b_n}$. 

The result is a path $v\mapsto \tilde\gamma_v$ from $\tilde\gamma$ to $\tilde\gamma'$ which we claim is continuous. Continuity at any $v_0$ which is in the interior of one of the inserted intervals $I_n$ is provided by Lemma~\ref{lem:cts pivot}. If $v_0$ is on the boundary of some $I_n$ and $v$ approaches $v_0$ from inside $I_n$, Lemma~\ref{lem:cts pivot} again applies. Otherwise, $\tilde \gamma_{v_0}$ is in $\Sing$ and as $v\to v_0$, $\tilde \gamma_v$ approaches $\tilde \gamma_{v_0}$ from a side on which $\tilde \gamma_{v_0}$ always turns with angle $\pi$. In this case, let $\eps>0$ be given. Since there are only finitely many cone points in any compact region of $\tilde \SSS$, for $v$ sufficiently close to $v_0$ there are no cone points in the convex hull of $\tilde\gamma_{v_0}[-T(\eps),T(\eps)]$ and $\tilde\gamma_v[-T(\eps),T(\eps)]$. Perhaps making $v$ even closer to $v_0$, this convex hull is a rectangle with width $<\frac{\eps}{2}$. Then, by Lemma~\ref{lem:shadow in S}, $d_{G\tilde\SSS}(\tilde \gamma_v,\tilde\gamma_{v_0})<\eps$, proving continuity at $v_0$.

Finally, we claim that $d_{G\tilde \SSS}(\tilde \gamma, \tilde \gamma_v)$ is non-decreasing. Let $a<b$ be in $[0,2]$. If $a,b\in I_n$, $d_{G\tilde \SSS}(\tilde \gamma, \tilde \gamma_a) \leq d_{G\tilde \SSS}(\tilde \gamma, \tilde \gamma_b)$ by Lemma~\ref{lem:cts pivot}. Therefore, to prove the distance is non-decreasing in general we just need to show $d_{G\tilde \SSS}(\tilde \gamma, \tilde \gamma_a) \leq d_{G\tilde \SSS}(\tilde \gamma, \tilde \gamma_b)$ when $a$ and $b$ are close and $a$ is the lower endpoint of some $I_n$ or is in the complement of the $\{I_n\}$. In either case, $\tilde \gamma_a$ is a singular geodesic which makes angle $\pi$ at any cone points it encounters on the side away from $\tilde \gamma$. For each fixed $t$ consider the geodesic segment $c_{v,t}=[\tilde \gamma(t), \tilde \gamma_v(t)]$ and how it varies with $v$. We claim the length of $c_{v,a}$ is at most the length of $c_{v,t}$ for small enough $t>a$, which together with the formula for $d_{G\tilde \SSS}$ will establish the desired result. As $v$ increases from $a$, $\tilde\gamma_v(t)$ will move along a geodesic path perpendicular to $\tilde \gamma_a$ on the side of $\tilde \gamma_a$ away from $\tilde \gamma$. Indeed, for all $b>a$ small enough that no cone points are in the convex hull of $\tilde\gamma_a[0,t]$ and $\tilde \gamma_b[0,t]$, by construction, $\tilde\gamma_b[0,t]$ will simply be the translation of $\tilde \gamma_a[0,t]$ across an embedded Euclidean rectangle. Take such a $b>a$. Then, consider the geodesic triangle with sides $c_{a,t}$, $c_{b,t}$ and $[\tilde\gamma_a(t),\tilde \gamma_b(t)]$. Since $[\tilde\gamma_a(t),\tilde \gamma_b(t)]$ is perpendicular to $\tilde \gamma_a$ on the side away from $\tilde \gamma$ and $c_{a,t}$ hits $\tilde \gamma_a(t)$ from the side towards $\tilde \gamma$, the angle between $c_{a,t}$ and $[\tilde\gamma_a(t),\tilde \gamma_b(t)]$ is at least $\frac{\pi}{2}.$ By comparison with a Euclidean triangle and the CAT(0) property, $c_{b,t}$ is longer than $c_{a,t}$ giving the desired result.

Therefore, $\tilde \gamma_v$ is in the same connected component of $B(\Sing,\delta)$ for all $v$ for this `slide' move.

\

We return now to our construction of $\tilde c$.  For any $r\in [L, t-L]$, we can reach $g_r$ via the following series of the moves noted above. First, we move $\tilde \gamma \to g_{t/2}\tilde\gamma$ by geodesic flow. Second, we slide $g_{t/2}\tilde\gamma$ down across $R$ (if $R$ has nonzero height) to a geodesic in the orbit of $\tilde \gamma_0'$ using our `slide' move. We break this move down into a sequence of small `slide' moves between geodesics $\tilde \gamma_{v_n}$ in $\Sing$. Since $t>3L$ and $L\geq 2T(\delta)$ if we choose $v_n$ so that $\tilde \gamma_{v_n}\cap R$ and $\tilde \gamma_{v_{n+1}}\cap R$ are within $\delta/2$ vertically in $R$, by Lemma~\ref{lem:shadow in S}, $d_{G\tilde\SSS}(\tilde \gamma_{v_n},\tilde \gamma_{v_{n+1}})<\delta$. Therefore, this series of moves stays in the same connected component of $B(\Sing,\delta)$. Finally, we apply a series of `pivot' moves and the geodesic flow to get to $g_rc$ via the geodesics $\tilde \gamma'_i$ introduced in our construction above. Our work in the construction showed that all the `pivot' moves involved are between geodesics within $\delta$ of one another. Therefore, in total, we have a continuous path from $\tilde \gamma$ to $g_r\tilde c$ in $B(\Sing,\delta)$, completing the proof of Proposition \ref{prop:regular shadow}.
\end{proof}

The second step in the argument for the pressure gap is to prove the following Lemma, which uniformly controls how many geodesics in $\Sing$ can have image under $\Pi_t$ near to a fixed geodesic. Recall that 
\[ d_{\GS, t}(\gamma_1, \gamma_2) = \max_{s\in[0,t]}d_{\GS}(g_s\gamma_1, g_s\gamma_2), \]
and that a subset of $\GS$ is $(t,2\eps)$-separated if its members are pairwise distance at least $2\eps$ apart with respect to $d_{\GS,t}$.

\begin{lemma}[Compare with Prop. 8.2 in \cite{BCFT}]\label{lem:mult bound}
For all $\eps>0$, there exists some $C(\eps)>0$ such that if $E_t \subset \Sing$ is a $(t, 2\eps)$-separated set for some $t>3L$, then for any $w\in GS$, 
\[ \# \{ \gamma \in E_t \mid d_{\GS,t}(w, \Pi_t(\gamma))<\eps \} \leq C. \]
\end{lemma}

\begin{proof}
It is sufficient to prove the result in $G\tilde S$.

Let $d_0$ be as in Lemma~\ref{lem:no big flat}(a). Fix $w\in G\tilde S$, and let $\eps>0$, $t>3L$, and $E_t$ be given. 

Suppose that $d_{\GS,t}(w,c)<\eps$. Then by definition $d_{G\tilde \SSS}(g_r w, g_r c)<\eps$ for all $r\in[0,t]$. By Lemma~\ref{lem:closeness in S and GS}, that $d_{G\tilde \SSS}(w,c)<\eps$ implies $d_{\tilde \SSS}(w(0), c(0))<2\eps$ and that $d_{G\tilde \SSS}(g_t w,g_t c)<\eps$ implies $d_{\tilde \SSS}(w(t), c(t))<2\eps$.

By Proposition~\ref{prop:regular shadow}, any geodesic $c$ in $\Pi_t(\Sing)$ has $c(0)\in\widetilde{Con}$ and $c(t+t')\in\widetilde{Con}$ for some $|t'|<4d_0$. Using what we noted above, the cone point at $c(0)$ must be within $d_{\tilde \SSS}$-distance $2\eps$ of $w(0)$ and the cone point at $c(t+t')$ must be within $d_{\tilde \SSS}$-distance $2\eps+4d_0$ of $w(t)$.

As $\SSS$ is compact and $Con$ is a discrete subset, for any $R>0$, $N_R=\max_{p\in\tilde\SSS}\#\{ \widetilde{Con} \cap B_R(p) \}$ is finite. Let $C_1(\eps) = N_{2\eps}N_{2\eps+4d_0}$. As specified in Proposition~\ref{prop:regular shadow}, any element $c$ of $\Pi_t(\Sing)$ is entirely determined by the cone points $c(0)$ and $c(t+t')$. Thus, there are at most $C_1(\eps)$ elements $c\in \Pi_t(\Sing)$ with $d_{G\tilde \SSS}(w,c)<\eps$.

Now we want to bound $\#\{\gamma \in E_t \mid \Pi_t(\gamma)=c \}$ for any $c\in \Pi_t(\Sing)$. For $\gamma\in E_t$, the construction of $\Pi_t(\gamma)$ shows that $d_{\tilde \SSS}(\gamma(0),c(0))<2d_0$ and $d_{\tilde \SSS}(\gamma(t), c(t+t'))<2d_0$. Therefore, $\gamma(-T(\eps))\in B(c(0),2d_0+T(\eps))$ and $\gamma(t+T(\eps)) \in B(c(t+t'), 2d_0+T(\eps))$, where $T(\eps)$ is as in Lemma~\ref{lem:shadow in S}. Let $P$ be an $\frac{\eps}{8}$-spanning set for $B(c(0),2d_0+T(\eps))$ with respect to $d_{\tilde \SSS}$ and $Q$ an $\frac{\eps}{8}$-spanning set for $B(c(t+t'), 2d_0+T(\eps))$ with respect to $d_{\tilde \SSS}$. By the compactness of $\SSS$, there exists some $C_2(\eps)$ such that $\#P$ and $\#Q$ are bounded above by $C_2(\eps)$. For each $(p,q)\in P\times Q$, extend $[p,q]$ to a geodesic $\eta_{p,q}$ with $\eta_{p,q}(-T(\eps)) = p$.

Since $P$ and $Q$ are $\frac{\eps}{8}$-spanning, there exist $(p,q)\in P\times Q$ such that $d_{\tilde \SSS}(\gamma(-T(\eps)), p)<\frac{\eps}{8}$ and $d_{\tilde \SSS}(\gamma(t+T(\eps)), q)<\frac{\eps}{8}$. We immediately have that $d_{\tilde\SSS}(\gamma(-T(\eps)), \eta_{p,q}(-T(\eps)))<\frac{\eps}{8}$. In addition, $\gamma[-T(\eps), t+T(\eps)]$ and $[p,q]$ are geodesic segments whose endpoints are each less than $\frac{\eps}{8}$ apart. Since geodesic segments in $\tilde \SSS$ minimize length, the length of $[p,q]$ is within $\frac{\eps}{4}$ of $t+2T(\eps)$, the length of $\gamma[-T(\eps),t+T(\eps)]$. Therefore we also have
\begin{align}
	d_{\tilde \SSS}(\gamma(t+T(\eps)),\eta_{p,q}(t+T(\eps))) & \leq d_{\tilde \SSS}(\gamma(t+T(\eps)),q)+d_{\tilde \SSS}(q,\eta_{p,q}(t+T(\eps)) \nonumber \\ 
			& < \frac{\eps}{8} + \frac{\eps}{4} < \frac{\eps}{2}. \nonumber
\end{align}

Using convexity of the distance between geodesics in a CAT(0) space and our bounds on the distances between the pairs of endpoints, we have 
\[ d_{\tilde \SSS}(\gamma(r), \eta_{p,q}(r))<\frac{\eps}{2} \ \ \mbox{ for all } \ \ r\in [-T(\eps), t+T(\eps)].\]
Then by Lemma~\ref{lem:shadow in S}, $d_{\tilde G \SSS}(g_r\gamma, g_r\eta_{p,q})<\eps$ for all $r\in [0, t],$ or, equivalently, $d_{G\tilde \SSS,t}(\gamma,\eta_{p,q})<\eps$.

We can conclude that $\#\{\gamma \in E_t \mid \Pi_t(\gamma)=c \} \leq \#\{ \eta_{p,q}\} \leq C_2(\eps)^2.$ Indeed, if there are more than $C_s(\eps)^2$ elements in $E_t$ which have image $c$ under $\Pi_t$, then some two of them must both be within $d_{G\tilde \SSS,t}$-distance $\eps$ of the same $\eta_{p,q}$ and hence less than $2\eps$ apart with respect to $d_{G\tilde \SSS,t}$, contradicting the fact that $E_t$ is $(t,2\eps)$-separated.

Putting these estimates together, $\# \{ \gamma \in E_t \mid d_{\GS,t}(w, \Pi_t(\gamma))<\eps \} \leq C_1(\eps)C_2(\eps)^2,$ completing the proof.
\end{proof}

The third step of the argument closely follows \cite{BCFT}, as we now outline. First, by Lemmas 4.1 and 4.2 of \cite{Climenhaga-Thompson}, for any $\eps>0$ and $t>0$,
\begin{equation}\label{lower bound}
 \sup\left\{ \sum_{\gamma\in E} e^{\sup_{\xi\in B_t(\gamma,\eps)}\int_0^t \phi(g_r\xi) dr} \biggm| E\subset \Sing \mbox{ is } (t,\eps)\mbox{-separated} \right\} \geq e^{tP(\Sing, 2\eps, \phi)}.
\end{equation}
To apply this fact from \cite{Climenhaga-Thompson} here, we just need to recall that $\Sing$ is compact (noted in Definition~\ref{defn:Sing}).

We now use the fact that $\phi$ is locally constant on a neighborhood of $\Sing$. For sufficiently small $\eps$, the left-hand side of the inequality above is equal to
\begin{equation}\label{equality to Lambda}
 \Lambda(\Sing,\phi,\eps,t) := \sup\left\{ \sum_{\gamma\in E} e^{\int_0^t \phi(g_r\gamma) dr} \biggm| E\subset \Sing \mbox{ is } (t,\eps)\mbox{-separated} \right\}. 
\end{equation}
Combining \eqref{lower bound} and \eqref{equality to Lambda} and using the fact that $g_t$ is entropy-expansive (Lemma~\ref{lem:entropy-expansive}) exactly as in \cite{BCFT}, for sufficiently small $\eps$,
\begin{equation}\label{eqn:partition sum}
    \Lambda(\Sing,\phi,\eps,t) \geq e^{tP(\Sing,\phi)}.
\end{equation}

Fix $0<\eta<\frac{\eta_0}{2}$ where $\eta_0$ is from Lemma \ref{lem:no big flat}(b). Note that $\Reg(\eta)$ has nonempty interior. Pick $\delta>0$ small enough that $\Lambda(\Sing,\phi,2\delta,t) \geq e^{tP(\Sing,\phi)}$, $\phi$ is locally constant on $B(\Sing,\delta)$, and by Lemma~\ref{lem:lambda small}, $\lambda(\gamma)<\eta$ for all $\gamma\in B(\Sing, 2\delta)$. Then we proceed exactly as in \cite{BCFT}, invoking Proposition~\ref{prop:regular shadow} as a direct replacement of their Theorem 8.1 and Lemma~\ref{lem:mult bound} as a direct replacement for their Proposition 8.2. The argument produces the following Lemma.

\begin{lemma}[Lemma 8.4 in \cite{BCFT}]\label{lem:partition sum}
For sufficiently small $\delta>0$, there is a $(t,2\delta)$-separated set $E_t$ in $\Sing$ such that there is a $(t,\delta)$-separated set $E_t''\subset \Pi_t(E_t)$ satisfying
\[ \sum_{w\in E_t''} e^{\inf_{u\in B_t(w,\delta)} \int_0^t \phi(g_s u)ds} \geq \beta e^{tP(\Sing,\phi)}\]
where $\beta=\frac{1}{C}e^{-6L\Vert\phi\Vert}$, and $C$ is as in Lemma~\ref{lem:mult bound}.
\end{lemma}

\begin{proof}
The only minor change needed in substituting our Proposition~\ref{prop:regular shadow} for their Theorem 8.1 is to note that our condition on $t$ is that it be $>3L$, whereas theirs is that it be $>2L$. This gives us $\beta=\frac{1}{C}e^{-6L\Vert\phi\Vert}$ instead of $\beta=\frac{1}{C}e^{-4L\Vert\phi\Vert}$. This results in merely cosmetic changes to the rest of the argument in \cite{BCFT}.
\end{proof}

Note that $\{(w,t): w\in E_t''\}$ is in $\mathcal{G}^{4d_0}(\eta)$, using the notation of Definition~\ref{defn: good bounded shift}.

The final step in the argument is to use specification to string together orbit segments from $E_t''$ in many different orders so as to produce a large collection of long orbit segments which together produce more pressure than $P(\Sing,\phi)$. In \cite{BCFT} this is undertaken in Section 8.4, and at this point the argument is almost entirely dynamical. It uses the estimate of Lemma~\ref{lem:partition sum} together with strong specification for $\mathcal{G}^{4d_0}(\eta)$ as given by Corollary~\ref{cor:expanded weak specification}. The one geometric piece of information used is that $\lambda(\gamma)<\eta$ for all $\gamma\in B(\Sing, 2\delta)$. Hence, we assumed this when choosing $\delta$ above, invoking Lemma~\ref{lem:lambda small}. This completes the proof of Theorem~\ref{thm: locally constant}.\qed

Applying Theorem~\ref{thm: locally constant} with $\phi=0$ gives the following.

\begin{corollary}\label{cor:top entropy gap}
$h_{top}(g_t|_{\Sing}) < h_{top}(g_t).$
\end{corollary}

With the pressure gap condition for such potentials in hand we briefly note a second class of potentials for which it holds. Proposition 4.7 of \cite{Ca20} notes that if the pressure gap $P(\Sing,\phi)<P(\phi)$ holds for $\phi$, then for any function sufficiently close to $\phi$ (specifically with $2\Vert \phi - \psi\Vert < P(\phi)-P(\Sing,\phi)$) and any constant $c$, $P(\Sing,\psi+c)<P(\psi+c)$. Applying this to the locally constant functions $\phi$ discussed in this section gives us a further class of potentials with a pressure gap. Applying it with $\phi=0$ gives us one class of particular note:

\begin{corollary}\label{cor: nearly constant}
If $\psi$ is a continuous potential with $\|\psi\|<\frac{1}{2}\left(h_{top}(g_t) - h_{top}(g_t|_{\Sing})\right)$, where $h_{top}$ is the topological entropy, then $P(\Sing,\psi)<P(\psi)$.
\end{corollary}

%

\section{Equilibrium states are Limits of Weighted Periodic Orbits}\label{sec: weighted}

We can show that weighted periodic orbits equidistribute to the equilibrium states we have constructed, following a method of \cite{BCFT}. Throughout this section, we write $\GGG^M := \GGG^M(\eta)$ (see Definition~\ref{defn: good bounded shift}) as we will work with a fixed $\eta$ throughout.

Define the equivalence class of a closed geodesic $[\gamma]$ to be all geodesics $\eta\in\GS$ for which $\gamma = g_t\eta$ for some $t\in\mathbb{R}$. Then let $\operatorname{Per}_R[Q-\delta, Q]$ be the set of equivalence classes of regular closed geodesics with period in $[Q-\delta,Q]$. Now consider such a regular closed geodesic and define $\mu_\gamma$ to be the normalized Lebesgue measure supported on $\gamma$ and $\Phi(\gamma) = \int_{0}^{\ell(\gamma)}\phi(g_u\gamma)\,du$. These definitions agree for all representatives of an equivalence class, so we define $\mu_{[\gamma]} = \mu_\gamma$ and $\Phi([\gamma]) = \Phi(\gamma)$. We consider the weighted sum
$$\mu_{Q,\delta} = \frac{1}{\Lambda_R(Q,\delta,\phi)}\sum_{[\gamma]\in\operatorname{Per}_R[Q-\delta,Q]}e^{\Phi([\gamma])}\mu_{[\gamma]},$$
where $\Lambda_R(Q,\delta,\phi) = \sum\limits_{[\gamma]\in\operatorname{Per}[Q-\delta,Q]}e^{\Phi([\gamma])}$ is our normalizing constant. When $\lim\limits_{Q\rightarrow\infty}\frac{1}{Q}\log\Lambda_R(Q,\delta,\phi)$ exists, it can be thought of as the pressure of closed saddle connection paths, and we write it as $P_{R,\delta}(\phi)$.

\begin{theorem}\label{thm: equidistribute}
We use the notation above. Let $\phi$ be a H\"{o}lder potential with $P(\Sing,\phi) < P(\phi)$, and let $\mu$ be the unique equilibrium state for $\phi$. Then, for all $\delta>0$, $P_{R,\delta}(\phi) = P(\phi)$ and in the weak-* topology we have $\lim\limits_{Q\rightarrow\infty}\mu_{Q,\delta}=\mu$.
\end{theorem}

\begin{remark*}
Note that this provides a way to identify interesting potentials, by considering geometrically relevant ways to weight closed geodesics. For instance, one could potentially try to identify a continuous function that weights $\gamma$ by the number of conical points it turns at.
\end{remark*}

We first prove a lemma that will be necessary throughout this section.

\begin{lemma}\label{lem:separated geodesics}
Let $2\eps$ be less than the injectivity radius of $S$. For all $Q \gg \delta > 0$, any set of representatives of the equivalence classes in $\operatorname{Per}_R[Q-\delta,Q]$ is $(Q,\eps)$-separated.
\end{lemma}

\begin{proof}
Consider $[\gamma_1],[\gamma_2]\in \operatorname{Per}_R[Q-\delta, Q]$, and let $\gamma_1,\gamma_2$ be representatives. Furthermore, suppose $d_{\GS}(g_t\gamma_1,g_t\gamma_2) < \eps$ for all $t\in [0,Q]$. By Lemma~\ref{lem:closeness in S and GS}, $d_{\SSS}(\gamma_1(t),\gamma_2(t)) < 2\eps$ for all $t\in [0,Q]$. By our choice of $\eps$, these geodesics are freely homotopic and represent the same element $g$ of the fundamental group. Letting $\tilde{\gamma}_i$ be lifts of $\gamma_i$, we have that both $\tilde{\gamma}_1$ and $\tilde{\gamma}_2$ are axes of $g$. By \cite[Theorem II.6.8]{bh}, $\tilde{\gamma}_1$ and $\tilde{\gamma}_2$ are parallel, and so they bound a flat strip by the Flat Strip Theorem. This contradicts the assumption that $\gamma_1$ and $\gamma_2$ are regular.
\end{proof}

We have the following proposition, which follows from the proof of Variational Principle found in \cite[Theorem 9.10]{Wa} because $\operatorname{Per}_R[Q-\delta,Q]$ is $(Q,\eps)$-separated for all sufficiently small $\eps$:

\begin{proposition}\label{weighted limit}
	If $\mu$ is the unique equilibrium state for $\phi$, then for all $\delta > 0$ such that $\lim\limits_{Q\to\infty}\frac{1}{Q}\log\Lambda_R(Q,\delta,\phi) = P(\phi)$, we have $\lim\limits_{Q\to\infty}\mu_{Q,\delta} = \mu$.
\end{proposition}

In order to apply this proposition, we need to establish a growth rate for $\Lambda_R(Q,\delta,\phi)$ for all sufficiently small $\delta > 0$, which is done in Propositions~\ref{prop: lower bound} and \ref{prop: upper bound} below.

First, we show that the growth rate for $\Lambda_R(Q,\delta,\phi)$ is fast enough. In order to do this, we need to be able to approximate $(\gamma,t)\in\GGG^M$ by closed geodesics of a bounded length. This is encapsulated in the following proposition.

\begin{proposition}\label{periodic approximation}
For all $\delta > 0$, there exists $T'$ such that for all $(\gamma,t)\in\GGG^M$ with $t > \frac{\theta_0}{\eta} + 2M$, there is some regular closed geodesic $\xi$ with period in $[t+T'-\delta,t+T']$ such that $d_{\GS}(g_u\gamma,g_u\xi) < \delta$ for all $u\in [0,t]$.
\end{proposition}

\begin{proof}
First, we explain how to obtain the statement of the proposition for $(\gamma,t)\in\GGG$. 
Let $\delta > 0$, and let $\hat\tau$ be the specification constant for $\GGG$ with shadowing scale $\frac{\delta}{8}$ (see Proposition~\ref{specification}). We will show that $T' = \hat\tau + \frac{\delta}{2}$ satisfies the requirements of the Proposition. Let $(\gamma,t)\in \GGG$, and let $\xi$ be a geodesic guaranteed by specification which shadows $(\gamma,t)$ twice in succession. Now recall from Proposition \ref{specification} that there exists a closed interval $I\supset [\frac{\theta_0}{2\eta},t-\frac{\theta_0}{2\eta}]$ such that $\xi$ contains two copies of $\gamma(I)$.
In other words, there exist $r_1,r_2 > 0$ such that $\xi(r_i+r) = \gamma(r)$ for all $r\in I$, where $i\in\{1,2\}$. Thus, we can choose $\xi$ to be a closed geodesic, and observe that its length is given by $r_2-r_1$. Now, since $d_{\GS}(g_{\frac{\theta_0}{2\eta}}\xi,g_{\frac{\theta_0}{2\eta}}\gamma)\leq \frac{\delta}{8}$ by Proposition~\ref{specification}, we can apply Lemma~\ref{lem:closeness in S and GS} to show $d_{S}(\xi(\frac{\theta_0}{2\eta}),\gamma(\frac{\theta_0}{2\eta})) \leq \frac{\delta}{4}$. Thus, $|r_1| \leq \frac{\delta}{4}$. Similarly, considering the second copy of $(\gamma, t)$ that $\xi$ shadows, we have $d_{S}(\xi(\frac{\theta_0}{2\eta}+t+\hat\tau),\gamma(\frac{\theta_0}{2\eta})) \leq \frac{\delta}{4}$, and so $r_2 \in [\hat\tau+t-\frac{\delta}{4},\hat\tau+t+\frac{\delta}{4}]$. Hence, $\xi$ is a regular closed geodesic with length in $[\hat\tau+t-\frac{\delta}{2},\hat\tau+t+\frac{\delta}{2}]$. Taking $T' = \hat\tau+\frac{\delta}{2}$, we are done.
In order to adapt this argument to $\GGG^M$ for $\tau > 0$, note that we achieve specification for $\GGG^M$ by considering the specification constant for $\GGG$ at a smaller scale (which depends on $M$). (See Corollary~\ref{cor:expanded weak specification}.)
\end{proof}

To establish the desired growth rates on $\Lambda_R(Q,\delta,\phi)$, we need two technical counting results from \cite{Climenhaga-Thompson}. These results are used implicitly in the proof of Theorem~\ref{thm: to get weakly mixing}, and we do not provide a self-contained proof in the interest of concision. However, we do discuss why they hold in our setting.

As noted in Section~\ref{sec:outline}, the conditions that we check differ slightly from those used in \cite{Climenhaga-Thompson}. The only case where they are not immediately stronger conditions is the pressure estimate. In \cite{Climenhaga-Thompson}, the authors need to define the pressure of a discretized collection of orbit segments $P([\PPP]\cup[\SSS],\phi) < P(\phi)$. Because we use $\lambda$-decompositions, we do not need to consider the pressure of collections of orbit segments (this is the content of \cite[Lemma 3.5, Theorem 3.6]{CT19} and \cite[Proposition 4.2]{Ca20}). Instead, it suffices to show that $P\left(\bigcap_{t\in\mathbb{R}}g_t\lambda^{-1}(0),\phi\right) < P(\phi)$, which is precisely the condition $P(\Sing,\phi)< P(\phi)$.

The lemmas we will use are the following.

\begin{lemma}[{\cite[Lemma 4.12]{Climenhaga-Thompson}}]\label{lem: pressure of good segments}
There exist $C, \eps, M > 0$ such that for all $t > 0$, there exists a $(t,\eps)$-separated set $E_t$ with the following properties:
\begin{itemize}
    \item $\sum_{\gamma\in E_t}\exp\left(\int_0^t\phi(g_u\gamma)\,du\right)\geq Ce^{tP(\phi)}$
    \item $E_t\subset \{\gamma \in \GS \mid (\gamma,t)\in\GGG^M\}$.\qedhere
\end{itemize}
\end{lemma}

\begin{lemma}[{\cite[Lemma 4.11]{Climenhaga-Thompson}}]\label{lem:pressure bounds}
For all $\eps > 0$ sufficiently small, there exists a constant $D > 0$ such that for any $(t,\eps)$-separated set $E_t$, we have
$$\sum_{\gamma\in E_t}\exp\left(\int_0^t\phi(g_u\gamma)\,du\right) \leq De^{tP(\phi)}.$$
\end{lemma}

We are now ready to prove our growth rates.

\begin{proposition}\label{prop: lower bound}
	For all $\delta > 0$ there exists a constant $\hat C$ such that 
	$$\Lambda_R(Q,\delta,\phi) \geq \frac{\hat C}{Q}e^{QP(\phi)}$$
	for all sufficiently large $Q$.
\end{proposition}

The proof of this proposition follows almost exactly the proof of the lower bound in \cite[Proposition 6.4]{BCFT}, replacing the use of \cite[Corollary 4.8]{BCFT} with Proposition~\ref{periodic approximation}. We include it here for completeness.

\begin{proof}

Let $C,\eps,M$ and $E_t$ be as in Lemma~\ref{lem: pressure of good segments}. Now, choose $\rho < \frac{\eps}{3}$ small enough that the Bowen property at scale $\rho$ holds on $\GGG^M$ (that this is possible follows immediately from the fact that $\GGG$ has the Bowen property). Then, by Proposition~\ref{periodic approximation}, there exists $T' > 0$ so that when $t > \frac{\theta_0}{\eta} + 2M$, there is an injective mapping from $E_t$ to a set $P_t$ of regular closed geodesics with periods in $[t+T'-\delta,t+T']$, i.e. for any $\xi\in P_t$ there exists $u\in [t+T'-\delta,t+T']$ such that $g_u\xi = \xi$. 
In particular, for all $\gamma\in E_t$, there exists $\xi\in P_t$ so that $d_{\GS}(g_u\xi,g_u\gamma) \leq \rho$ for all $u\in[0,t]$. Because the mapping is injective and $\phi$ has the Bowen property at scale $\rho$ on $\GGG^M$, it follows from Lemma~\ref{lem: pressure of good segments} that
	$$\sum_{\xi\in P_t}\exp\left(\int_{0}^{t}\phi(g_u\xi)\,du\right) \geq Ce^{-K}e^{tP(\phi)}$$
	for some constant $K$ independent of $t$. Now, writing $\Phi(\xi) = \int_{0}^{\ell(\xi)}\phi(g_u\xi)\,du$, we can then write
	$$\sum_{\xi\in P_t}\exp(\Phi(\xi))\geq \sum_{\xi\in P_t}\exp\left(\int_{0}^{t}\phi(g_u\xi)\,du - T'\lVert\phi\rVert\right) \geq Ce^{-(K+T'\lVert\phi\rVert)}e^{tP(\phi)}.$$
	At this point, we can almost relate this to $\Lambda_R(Q,\delta,\phi)$. However, there is a possibility that $\xi_1,\xi_2\in P_t$ both represent the same closed geodesic path, i.e., there exists $u$ so that $g_u\xi_1 = \xi_2$. As $P_t$ is $(t,\rho)$-separated and $d_{\GS}(\eta,g_u\eta) = u$, there are at most $\frac{t+T'}{\rho}$ such repetitions. Hence, if $Q\geq T$, by setting $Q=t+T'$, we have
	$$\Lambda_R(Q,\delta,\phi) \geq \left(\frac{\rho}{Q}\right)Ce^{-K}e^{-T'(\lVert\phi\rVert + P(\phi))}e^{QP(\phi)}.$$
\end{proof}

In order to see that the growth rate is not too large, we use Lemma~\ref{lem:separated geodesics} and Lemma~\ref{lem:pressure bounds}.

\begin{proposition}\label{prop: upper bound}
For all $\delta>0$ there exists a constant $D>0$ such that $$\Lambda_{R}(Q, \delta,\phi) \leq De^{\delta\lVert\phi\rVert}e^{QP(\phi)}$$ for all sufficiently large $Q$.
\end{proposition}

\begin{proof}
By Lemma~\ref{lem:separated geodesics}, any set of representatives of $\operatorname{Per}_R[Q-\delta,Q]$ is $(Q,\eps)$-separated for $\eps$ sufficiently small, and in particular, small enough to apply Lemma~\ref{lem:pressure bounds}. Now, given $[\gamma]\in\operatorname{Per}_R[Q-\delta,Q]$, observe that $\left|\Phi(\gamma) - \int_{0}^{Q}\phi(g_u\gamma)\,du\right| \leq \delta\lVert\phi\rVert$, because we know the period of $\gamma$ is at least $Q-\delta$. Consequently, it follows that for such an $\eps$, there exists $D > 0$ such that 
$$\Lambda_{R}(Q,\delta,\phi) \leq e^{\delta\lVert\phi\rVert}\sum_{[\gamma]\in\operatorname{Per}_R[Q-\delta,Q]}\exp\left(\int_0^Q\phi(g_u\gamma)\,du\right) \leq e^{\delta\lVert\phi\rVert}De^{QP(\phi)}.$$
\end{proof}

\begin{proof}[P\textbf{roof of Theorem \ref{thm: equidistribute}}]
Propositions \ref{prop: lower bound} and \ref{prop: upper bound} imply that
$$\lim\limits_{Q\to\infty}\frac{1}{Q}\log\Lambda_{R}(Q,\delta,\phi) = P(\phi).$$
By Proposition~\ref{weighted limit}, it follows that $\lim\limits_{Q\to\infty} \mu_{Q,\delta} = \mu$.
\end{proof}

%

\section{Acknowledgements}
We would like to thank the American Institute of Mathematics for their hospitality and support during the workshop ``Equilibrium states for dynamical systems arising from geometry" where most of the work was carried out and the SQuaRE ``Thermodynamic formalism for CAT(0) spaces" where further properties were considered. The first author was partially supported by NSF grant DMS-1954463. The third author was partially supported by NSF grant DMS-1547145. The last author was supported by Girls' Angle. We would like to thank Vaughn Climenhaga for comments on an earlier version of the paper, and an anonymous referee for numerous helpful comments.


\bibliography{bibliography}{}
\bibliographystyle{alpha}


\Addresses
\end{document}